\documentclass[11pt,letterpaper,reqno]{amsart}

\usepackage{graphicx}
\usepackage{color}
\usepackage{amsmath}
\usepackage{amsthm}
\usepackage{amssymb}
\usepackage{xspace}
\usepackage{epsfig}
\usepackage[dvipsnames]{xcolor}
\usepackage{soul}
\usepackage{xfrac}
\usepackage{enumitem}
\usepackage{hyperref}

\usepackage[foot]{amsaddr}

\usepackage[colorinlistoftodos,prependcaption,textsize=tiny]{todonotes}

\usepackage[normalem]{ulem}

\newcommand {\mm}[1] {\ifmmode{#1}\else{\mbox{\(#1\)}}\fi}

\newcommand{\ignore}[1]{}

\newsavebox{\smallProofsym}                 
\savebox{\smallProofsym}                               %
{
\begin{picture}(6,6)
\put(0,0){\framebox(6,6){}}
\put(0,2){\framebox(4,4){}}
\end{picture} 
}  

\makeatletter
\long\def\@makecaption#1#2{%
  \vskip\abovecaptionskip
  \sbox\@tempboxa{\small #1: #2}%
  \ifdim \wd\@tempboxa >\hsize
    \small #1: #2\par
  \else
    \global \@minipagefalse
    \hb@xt@\hsize{\hfil\box\@tempboxa\hfil}%
  \fi
  \vskip\belowcaptionskip}
\makeatother

\theoremstyle{plain}
\newtheorem{theorem}{\textbf{Theorem}}[section]
\newtheorem{definition}[theorem]{\textbf{Definition}}
\newtheorem{corollary}[theorem]{\textbf{Corollary}}
\newtheorem{lemma}[theorem]{\textbf{Lemma}}

\newtheorem*{supermaintheorem*}{Main Theorem}

\newtheorem*{supermaincorollary*}{Main Corollary}

\newtheorem{question}{Question}

\newtheorem{conjecture_new}[question]{Conjecture}

\newcommand{\Rspace}        {\mm{{\mathbb R}}}

\newcommand{\Fcal}          {\mm{{\mathcal F}}}

\newcommand{\Delaunay}[2]   {\mm{{\rm Del}_{#1}{({#2})}}}

\newcommand{\Label}[1]      {\mm{[{#1}]}}
\newcommand{\Star}[1]       {\mm{{\rm st}({#1})}}
\newcommand{\white}[1]      {\mm{{\rm wh}({#1})}}
\newcommand{\White}[1]      {\mm{{\rm White}({#1})}}
\newcommand{\Black}[1]      {\mm{{\rm Black}({#1})}}
\newcommand{\Vector}[1]     {\mm{{\rm Vector}({#1})}}

\newcommand{\card}[1]       {\mm{{\#}{#1}}}
\newcommand{\conv}[1]       {\mm{{\rm conv\,}{#1}}}

\newcommand{\Skip}[1]       {}

\definecolor{blue-green}{rgb}{0.0, 0.87, 0.87}

\begin{document}

\title[Order-2 Delaunay Triangulations Optimize Angles]{Order-2 Delaunay Triangulations Optimize Angles}

\author{Herbert Edelsbrunner}
\address{Herbert Edelsbrunner, IST Austria (Institute of Science and Technology Austria), Kloster\-neu\-burg, Austria}
\email{herbert.edelsbrunner@ist.ac.at}
\thanks{Work by the first and third authors is partially supported by the European Research Council (ERC), grant no.\ 788183, by the Wittgenstein Prize, Austrian Science Fund (FWF), grant no.\ Z 342-N31, and by the DFG Collaborative Research Center TRR 109, Austrian Science Fund (FWF), grant no.\ I 02979-N35.  Work by the second author is partially supported by the Alexander von Humboldt Foundation}

\author{Alexey~Garber}
\address{Alexey~Garber, The University of Texas Rio Grande Valley, Bronwsville, TX, USA}
\email{alexey.garber@utrgv.edu}

\author{Morteza Saghafian}
\address{Morteza Saghafian, IST Austria (Institute of Science and Technology Austria), Kloster\-neu\-burg, Austria}
\email{morteza.saghafian@ist.ac.at}

\date{\today}
\subjclass{05B45, 52C20, 68R05}
\keywords{Triangulations, higher order Delaunay triangulations, hypertriangulations, angle vectors, optimality}

\begin{abstract}
The \emph{local angle property} of the (order-$1$) Delaunay triangulations of a generic set in $\Rspace^2$ asserts that the sum of two angles opposite a common edge is less than $\pi$.
  This paper extends this property to higher order and uses it to generalize two classic properties from order-$1$ to order-$2$:
    (1)~among the complete level-$2$ hypertriangulations of a generic point set in $\Rspace^2$, the order-$2$ Delaunay triangulation lexicographically maximizes the sorted angle vector;
    (2)~among the maximal level-$2$ hypertriangulations of a generic point set in $\Rspace^2$, the order-$2$ Delaunay triangulation is the only one that has the local angle property.
  We also use our method of establishing (2) to give a new short proof of the angle vector optimality for the (order-1) Delaunay triangulation.
  For order-$1$, both properties have been instrumental in numerous applications of Delaunay triangulations, and we expect that their generalization will make order-$2$ Delaunay triangulations more attractive to applications as well.
\end{abstract}

\maketitle

\section{Introduction}
\label{sec:1}

This paper is motivated by the desire to generalize optimal properties from order-$1$ to higher-order Delaunay triangulations.
The classic \emph{(order-$1$) Delaunay triangulation} (also called \emph{Delaunay mosaic}) of a finite point set was introduced in 1934 by Boris Delaunay (also Delone).
It is the edge-to-edge tiling whose polygons satisfy the \emph{empty circle criterion} \cite{Del34}: each polygon is inscribed in a circle and all other points lie strictly outside this circle.
In the henceforth considered generic case, all polygons are triangles.
The criterion implies that for an edge shared by two triangles, the sum of the two angles opposite to the edge is less than $\pi$.
If a triangulation satisfies this criterion for every edge shared by two triangles, then we say the triangulation has the \emph{local angle property}.
Recognizing the potential of this type of triangulation for applications, Lawson in 1977 turned the empty circle criterion into an iterative algorithm that converts any triangulation of a given set of $n$ points in $\Rspace^2$ into the Delaunay triangulation using at most $O(n^2)$ edge-flips \cite{Law77}.
The correctness of this algorithm implies that the Delaunay triangulation is the only triangulation of the given set that has the local angle property.
Using Lawson's algorithm as a proof technique, Sibson proved in 1978 that among all triangulations of a finite generic point set in $\Rspace^2$, the Delaunay triangulation lexicographically maximizes the vector whose components are the angles inside the triangles sorted in non-decreasing order \cite{Sib78}.
We call this the \emph{sorted angle vector} of the triangulation.

\smallskip
The dual approach to the same topic predates the invention of the Delaunay triangulation.
In 1908, Georgy Voronoi published seminal papers on what today is called the \emph{Voronoi tessellation} \cite{Vor08}.
Given a finite set in $\Rspace^2$, this tessellation contains a (convex) region for each point in the set, such that the points in the region are at least as close to the generating point as to any other point in the set.
The Delaunay triangulation and the Voronoi tessellation of the same points are dual to each other: there is an incidence-preserving dimension-reversing bijection between the regions, edges, vertices of the tessellation and the vertices, edges, polygons of the triangulation.

In the mid 1970s, Shamos and Hoey~\cite{ShHo75} and Fejes T\'oth~\cite{Fej76} independently generalized this concept to the \emph{order-$k$ Voronoi tessellation}, which contains a (possibly empty) region for each subset of size $k$, such that the points in the region are at least as close to each one of the $k$ defining points as to any of the $n-k$ other points.
In 1982, Lee~\cite{Lee82} gave an incremental algorithm for computing these tessellations, and in 1990, Aurenhammer~\cite{Aur90} showed that there is a natural dual, which we refer to as the \emph{order-$k$ Delaunay triangulation}: each vertex is the average of a collection of $k$ points with non-empty region, and the triangles are formed by connecting two vertices with a straight edge if the corresponding two regions share an edge in the order-$k$ Voronoi tessellation.
The special case in which $k = n-1$ is closely related to the \emph{farthest-point Delaunay triangulation}: its vertices are the extreme points of the set (the convex hull vertices), and two vertices are connected by a straight edge if the regions in the order-$(n-1)$ Voronoi tessellation that correspond to the complementary $n-1$ points of the two vertices share a common edge.
In 1992, Eppstein~\cite{Epp92} proved an extension of Sibson's result: among all triangulations of the convex hull vertices, the farthest-point Delaunay triangulation lexicographically minimizes the sorted angle vector.

\smallskip
With the exception of Eppstein's result---which is specific to the farthest-point Delaunay triangulation---there is a paucity of optimality properties known for higher-order Delaunay triangulations, which we end with three inter-related contributions:
\begin{itemize}
  \item[I:] we extend the local angle property from order-$1$ to order-$k$, for $1 \leq k \leq n-1$, and show that the order-$k$ Delaunay triangulation has this property;
  \item[II:] we prove that among all complete level-$2$ hypertriangulations of a finite generic set in $\Rspace^2$, the order-$2$ Delaunay triangulation lexicographically maximizes the sorted angle vector;
  \item[III:] we show that among all maximal level-$2$ hypertriangulations of a finite generic set in $\Rspace^2$, the order-$2$ Delaunay triangulation is the only one that has the local angle property.
\end{itemize}
For ordinary triangulations, the proofs of the properties analogous to II and III follow from the existence of a sequence of edge-flips that connects any initial (complete) triangulation to the (order-$1$) Delaunay triangulation, such that every flip lexicographically increases the sorted angle vector.
While the level-$2$ hypertriangulations are connected by flips introduced in \cite{EGGHS23}, there are cases in which every connecting sequence contains flips that lexicographically decrease the sorted angle vector; see Section \ref{sec:6}.
Without this tool at hand, the relation between the local angle property and the sorted angle vectors is unclear, and the proofs of Properties~II and III fall back to an exhaustive analysis of elementary geometric cases.

\smallskip
This paper is organized as follows.
Section~\ref{sec:2} provides information on the main background, including level-$k$ hypertriangulations (maximal, complete, and otherwise) and the aging function.
Section~\ref{sec:3} introduces our extension of the local angle property to order $k$ and, in Theorem~\ref{thm:order-k_Delaunay_triangulations_have_lap}, shows that the order-$k$ Delaunay triangulation has this property.
Section~\ref{sec:4} proves Property~II in Theorem~\ref{thm:angle_vector_optimality} and, in Theorem~\ref{thm:sibson}, gives a new short proof of Sibson's theorem  on angle vector optimality for the order-1 Delaunay triangulation \cite{Sib78}. 
It also discusses possible extensions to the class of maximal level-$2$ hypertriangulations and to levels beyond $2$.
Section~\ref{sec:5} proves Property~III in Theorem~\ref{thm:level-2_hypertriangulations_have_lap}, which extends it to order-$3$ for points in convex position in Theorem~\ref{thm:level-3_hypertriangulations_and_convex_position}.
Finally, Section~\ref{sec:6} concludes the paper with discussions of open questions and conjectures related to the geometry and combinatorics of Delaunay and more general hypertriangulations.

\section{Background}
\label{sec:2}

We follow the standard approach to points in general position used in the literature:  a finite set, $A \subseteq \Rspace^2$, is \emph{generic} if no three points are colinear and no four points are cocircular.

\subsection{Triangulations and Hypertriangulations}
\label{sec:2.1}

We first define the families of all triangulations and hypertriangulations of $A$, which include the order-$1$ and order-$k$ Delaunay triangulations discussed in Section~\ref{sec:3}.
We write $\conv{A}$ for the convex hull of the set $A$.
\begin{definition}[Triangulations]
  \label{def:triangulations}
  A \emph{triangulation}, $P$, of a finite $A \subseteq \Rspace^2$ is an edge-to-edge subdivision of $\conv{A}$ into triangles whose vertices are points in $A$.
  It is usually identified with the set of its triangles, so we write $P = \{t_1, t_2, \ldots, t_m\}$.
  The triangulation is \emph{complete} if every point of $A$ is a vertex of at least one triangle, \emph{partial} if it is not complete, and \emph{maximal} if there is no other triangulation of the same points that subdivides it.
\end{definition}
It is easy to see that a triangulation is maximal iff it is complete.
We nevertheless introduce both concepts because they generalize to different notions for hypertriangulations, which we introduce next.
For a set of $k$ points, $I$, we write $\Label{I} = \frac{1}{k} \sum_{x \in I} x$ for the average of the points and, assuming $a \not\in I$ and $J \cap I = \emptyset$, we write $\Label{Ia}$ and $\Label{IJ}$ for the averages of $I \cup \{a\}$ and $I \cup J$, respectively.
While $\Label{I}$ is a point, we sometimes think of it as the set $I$, in which case we call it a \emph{label}.
\begin{definition}[Hypertriangulations \cite{EGGHS23}]
  \label{def:hypertriangulations}
  Let $A \subseteq \Rspace^2$ be generic, $n = \card{A}$, $k$ an integer between $1$ and $n-1$, and $A^{(k)} = \{ \Label{I} \mid I \subseteq A, \card{I} = k \}$ the set of $k$-fold averages of the points in $A$.
  A \emph{level-$k$ hypertriangulation} of $A$ is a possibly partial triangulation of $A^{(k)}$ such that every edge with endpoints $\Label{I}$ and $\Label{J}$ satisfies $\card{(I \cap J)} = k-1$.
\end{definition}
See Figure~\ref{fig:complete-maximal} for two level-$2$ hypertriangulations of five points.
Observe that every triangulation of $A$ is a level-$1$ hypertriangulation of $A$, and vice versa, but for $k > 1$, only a subset of the triangulations of $A^{(k)}$ are level-$k$ hypertriangulations of $A$.
Note also that it is possible that a point can be written as the average of more than one subset of $k$ points in $A$: for example, the center of a square is the $2$-fold average of two pairs of diagonally opposite vertices.
If a level-$k$ hypertriangulation uses such a point as a vertex, then it can use only one of the possible labels.

\smallskip
An alternative approach to these concepts is via induced subdivisions; see \cite[Chapter 9]{Zie95} for details, including the definitions of induced subdivisions and tight subdivisions.
According to this approach, a triangulation of $A = \{a_1, a_2, \ldots, a_n\}$ is a tight subdivision of $\conv{A}$ induced by the projection $\pi \colon \Delta_n \to \Rspace^2$, in which $\Delta_n = \conv{\{e_1, e_2, \ldots, e_n\}} \subseteq \Rspace^n$ is the standard $(n-1)$-simplex, and $\pi(e_i) = a_i$, for $i=1, 2, \ldots, n$.
To generalize, Olarte and Santos~\cite{OlSa21} use the level-$k$ hypersimplex, $\Delta_n^{(k)}$, which is the convex hull of the $k$-fold averages of the $e_i$ in $\Rspace^n$, and define a level-$k$ hypertriangulation of $A$ as a tight subdivision of $A^{(k)}$ induced by the same projection $\pi$ restricted to $\Delta_n^{(k)}$.
In this setting, the constraint to use only one label for each vertex is implicit.

\subsection{The Aging Function}
\label{sec:2.2}

A triangle in a level-$k$ hypertriangulation can be classified into two types.
Letting $\Label{I}, \Label{J}, \Label{K}$ be its vertices, each the average of $k$ points, we say the triangle is
\begin{itemize}
  \item \emph{black}, if $\card{(I \cap J \cap K)} = k-2$;
  \item \emph{white}, if $\card{(I \cap J \cap K)} = k-1$.
\end{itemize}
In other words, vertices of black triangles are labeled $\Label{X\!ab}, \Label{X\!ac}, \Label{X\!bc}$, for some $X$ of size $k-2$, and vertices of white triangles are labeled $\Label{Y\!a}, \Label{Y\!b}, \Label{Y\!c}$, for some $Y$ of size $k-1$.
Our next definition allows for transformations from white to black triangles.
\begin{definition}[Aging Function]
  \label{def:aging_function}
  Letting $t$ be a white triangle with vertices $\Label{Y\!a}, \Label{Y\!b}, \Label{Y\!c}$, the \emph{aging function} maps $t$ to the black triangle, $F(t)$, with vertices $\Label{Y\!ab}, \Label{Y\!ac}, \Label{Y\!bc}$.
\end{definition}
For example, the black triangles in the level-$2$ hypertriangulations in Figure~\ref{fig:complete-maximal} are obtained by aging the triangles in a complete triangulation on the left and a partial triangulation on the right.
The aging function increases the level of the triangle by one, hence the name.
Correspondingly, the inverse aging function maps a black triangle to a white triangle one level lower.

\smallskip
To extend this definition to hypertriangulations, we say a level-$k$ hypertriangulation, $P_k$, \emph{ages} to a level-$(k+1)$ hypertriangulation, $P_{k+1}$, denoted $P_{k+1} = F(P_k)$,
if the aging function defines a bijection between the white triangles in $P_k$ and the black triangles in $P_{k+1}$.
Note however that the aging of $P_k$ is not unique as it says nothing about the white triangles of $P_{k+1}$.
This notion is useful to obtain structural results for the family of all level-$k$ hypertriangulations.
For example, \cite{EGGHS23} has shown that every level-$2$ hypertriangulation is an aging of a level-$1$ hypertriangulation.
For the special case in which the points are in convex position, \cite{Gal18} has extended this result to all levels, $k$.
However, for points in possibly non-convex position, there are obstacles to applying the aging function. 
An example of a level-$2$ hypertriangulation, $P_2$, for which $F(P_2)$ does not exist is given in \cite{EGGHS23, OlSa21}.
\begin{figure}[hbt]
  \centering
  \vspace{0.1in}
  \resizebox{!}{1.8in}{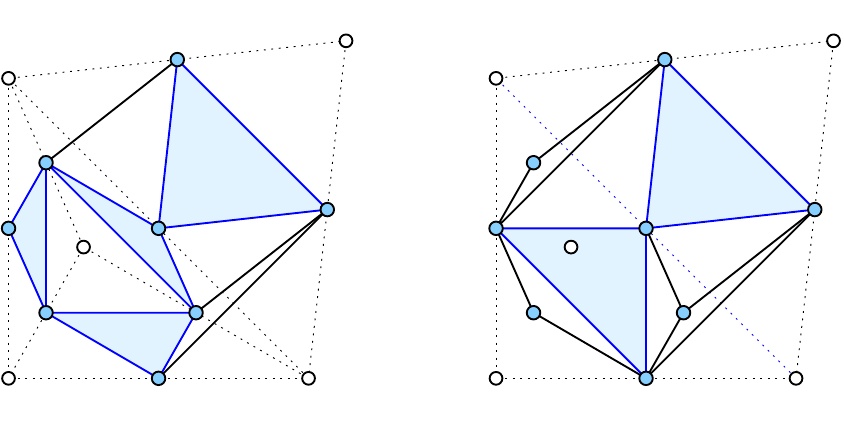}
  \vspace{-0.05in}
  \caption{Two level-$2$ hypertriangulations (drawn with \emph{solid} segments) obtained by aging the triangulations (drawn with \emph{dotted} segments) of the five \emph{white} points, $a,b,c,d,e$, with four triangles on the \emph{left} and two triangles on the \emph{right}.
  Black triangles are \emph{shaded} and white triangles are just \emph{white}.
  The total number of triangles is the same in both hypertriangulations.}
  \label{fig:complete-maximal}
\end{figure}

\smallskip
For later reference, we compile several results about the relation between level-$1$ and level-$2$ hypertriangulations obtained in \cite{EGGHS23}.
Given a vertex, $x$, in a triangulation, $P$, we define the \emph{star} of $x$ as the union of triangles that share $x$, denoted $\Star{P,x}$, and shrinking the star by a factor two toward $x$, we get $\Label{\Star{P,x},x} = \frac{1}{2} (\Star{P,x} + x)$, which is the set of midpoints between $x$ and any point $y \in \Star{P,x}$.
Observe that the shrunken star is contained in $\conv{A^{(2)}}$ iff $x$ is an interior vertex of $P$.
Indeed, $x$ necessarily belongs to the shrunken star, but if $x$ is a convex hull vertex, then $x$ lies outside $\conv{A^{(2)}}$.
The first two statements of the following lemma are Lemmas~4.1 and 4.3 in \cite{EGGHS23}, while the last statement is a reformulation of Lemma~4.2 in \cite{EGGHS23}.
We refer to \cite[Section~4]{EGGHS23} for proofs and additional details.

\begin{lemma}[Aging Function for Triangulations]
  \label{lem:aging_function_for_triangulations}
  Let $A \subseteq \Rspace^2$ be finite and generic, and recall that every level-$1$ hypertriangulation is a triangulation.
  \begin{itemize}
    \item For every level-$1$ hypertriangulation, $P$, of $A$, there exists a level-$2$ hypertriangulation, $P_2$, such that $P_2 = F(P)$.
    \item For every level-$2$ hypertriangulation, $P_2$, of $A$, there exists unique level-$1$ hypertriangulation, $P$, such that $P_2 = F(P)$.
    \item If $P_2 = F(P)$ and $x \in A$ is a vertex of $P$, then the union of white triangles in $P_2$ that have $x$ in all their vertex labels is $\Label{\Star{P,x},x} \cap \conv{A^{(2)}}$.
  \end{itemize}
\end{lemma}
Since $\Label{\Star{P,x},x} \cap \conv{A^{(2)}} \neq \Label{\Star{P,x},x}$ iff $x$ is a convex hull vertex, the third claim implies that for each interior vertex, $x$, scaled versions of the mentioned white triangles in $P_2$ tile the star of $x$ in $P$.

\subsection{Maximal and Complete Hypertriangulations}
\label{sec:2.3}

The Delaunay triangulation of a finite set is optimal among all complete triangulations, but not necessarily among the larger family of possibly partial triangulations of the set.
In this section, we introduce two families of level-$2$ hypertriangulations to which we compare the order-$2$ Delaunay triangulation.
\begin{definition}[Complete and Maximal Level-2 Hypertriangulations]
  \label{def:complete_and_maximal-level-2_hypertriangulations}
  Let $A \subseteq \Rspace^2$ be finite and generic.
  A level-$2$ hypertriangulation of $A$ is \emph{complete} if its black triangles are the images under the aging function of the triangles in a complete triangulation of $A$, and it is \emph{maximal} if no other level-$2$ hypertriangulation subdivides it.
\end{definition}
For example, the level-$2$ hypertriangulations in Figure~\ref{fig:complete-maximal} are both maximal, but only the one on the left is complete.
The notion of maximality extends to level-$k$ hypertriangulations for any relevant $k$. 
The situation with completeness is more subtle as it relies on the existence of compatible triangulations to which we can apply the aging function, and there are counterexamples to the existence from level $2$ to level $3$; see Figure~8 in \cite{EGGHS23}---which is based on Example~5.1 in \cite{OlSa21}---for a level-2 hypertriangulation with overlapping black triangles in level 3.

For $k=1$, a triangulation of a finite and generic set is complete iff it is maximal.
An easy way to see this is by counting the triangles in a possibly partial triangulation of $A \subseteq \Rspace^2$.
Write $H \subseteq A$ for the vertices of the convex hull of $A$, and set $n = \card{A}$ and $h = \card{H}$.
The vertex set of a partial triangulation can be any subset of $A$ that contains all points in $H$.
Let $m$ be the number of extra points, so the triangulation has $m+h$ vertices.
We can add $h-3$ (curved) edges to turn the triangulation into a maximally connected planar graph, which has $3(m+h)-6$ edges and $2(m+h)-4$ faces, including the outside.
Hence, the triangulation has $3(m+h)-6-(h-3) = 3m+2h-3$ edges and $2(m+h)-4 -(h-2) = 2m+h-2$ triangles.
For a complete triangulation, we have $m = n-h$ and therefore $2n-h-2$ triangles.
If a triangulation has fewer than this number, then its vertex set misses at least one point, which we can add by subdivision.
Hence, the triangulation is complete iff it is maximal.
The situation is slightly more complicated for level-$2$ hypertriangulations.
\begin{lemma}[Complete Implies Maximal]
  \label{lem:complete_implies_maximal}
  Let $A \subseteq \Rspace^2$ be finite and generic.
  Then any two maximal level-$2$ hypertriangulations have the same number of triangles, and every complete level-$2$ hypertriangulation is maximal.
\end{lemma}
\begin{proof}
  As before, let $H\subseteq A$ be the vertices of the convex hull of $A$.
  To prove the first claim, let $n = \card{A}$, $h = \card{H}$, and consider a level-$2$ hypertriangulation, $P_2$, aged from a possibly partial triangulation, $P$, with $m+h \leq n$ vertices.
  Note that $P$ has $2m+h-2$ triangles, so $P_2$ has the same number of black triangles.

  \smallskip
  To count the white triangles in $P_2$, we recall that each white region corresponds to the star of a vertex of $P$.
  If $a$ is a vertex in the interior of $\conv{A}$, then the white region is the shrunken star, $\Label{\Star{P,a},a}$.
  We modify $P_2$ so this is also true for each vertex, $b$, of $\conv{A}$.
  To this end, we consider all boundary edges of $P_2$ that connect vertices $a' = \Label{ba}$ and $c' = \Label{bc}$, and add the triangle $a'bc'$ to $P_2$.
  The number of thus added triangles depends on the convex hull of the midpoints of pairs but not on how this convex hull is decomposed into triangles.
  The benefit of this modification is that we now have exactly $m+h$ white regions, each a star-convex polygon, and each edge of $P$ contributes a vertex to exactly two of the white regions.
  Not forgetting the $h$ vertices added during the modification, this implies that the total number of edges of the $m+h$ white regions is $2 (3m+2h-3) + h = 6m+5h-6$.
  Every triangulation of a $j$-gon has $j-2$ triangles, so the total number of triangles in the white regions is $(6m+5h-6) - 2(m+h) = 4m+3h-6$.

  \smallskip
  We now turn our attention to the $n-h-m$ points of $A$ that are not vertices of $P$.
  Let $x$ be such a point and $abc$ the triangle in $P$ that contains $x$ in its interior.
  Hence, $\Label{xa}$ lies in the interior of $\Label{\Star{P,a},a}$, and similarly for $b$ and $c$.
  To maximally subdivide $P_2$, we thus add $3 (n-h-m)$ points in the interiors of the white regions, which increases the number of white triangles to $(4m+3h-6) + 6(n-h-m) = 6n-2m-3h-6$.
  Adding to this the $2m+h-2$ black triangles, we get a total of $6n-2h-8$ triangles.
  To get the number of triangles in this maximal triangulation, we still need to correct for the triangles added during the initial modification of $P_2$.
  But their number does not depend on $m$, so neither does the final triangle count.
  Hence, all maximal level-$2$ hypertriangulations of $A$ have the same number of triangles.

  \smallskip
  To get the second claim, observe that we have $m=0$ whenever $P_2$ is complete.
  Hence, we get the same number of triangles as just calculated, but without subdivision.
  It follows that $P_2$ is maximal.
\end{proof}

\section{The Local Angle Property}
\label{sec:3}

In this section, we define order-$k$ Delaunay triangulations as special level-$k$ hypertriangulations, introduce the local angle property for level-$k$ hypertriangulations, and show that the order-$k$ Delaunay triangulations have the local angle property.
This property specializes to the standard local angle property that characterizes (order-1) Delaunay triangulations as well as their constrained versions.

\subsection{Higher Order Delaunay Triangulations}
\label{sec:3.1}

We introduce the order-$k$ Delaunay triangulation of a finite set as a special level-$k$ hypertriangulation of this set; but see \cite{Aur90} for a more geometric definition.
\begin{definition}[Order-$k$ Delaunay Triangulation]
  \label{def:order-k_Delaunay_triangulation}
  Let $A \subseteq \Rspace^2$ be finite and generic, and $k$ an integer between $1$ and $\card{A}-1$.
  We construct a particular level-$k$ hypertriangulation of $A$:
  \begin{itemize}
    \item a black triangle with vertices $[Xab],[Xac],[Xbc]$ belongs to this hypertriangulation if $X \subseteq A$ is the set of points inside the circumcircle of $abc$, and $\card{X} = k-2$;
    \item a white triangle with vertices $[Ya],[Yb],[Yc]$ belongs to this hypertriangulation if $Y \subseteq A$ is the set of points inside the circumcircle of $abc$, and $\card{Y} = k-1$.
  \end{itemize}
  This hypertriangulation is called the \emph{order-$k$ Delaunay triangulation of $A$} and denoted $\Delaunay{k}{A}$.
\end{definition}
While it may not be obvious that the above triangles form a triangulation of $A^{(k)}$, it can be seen, for example, by lifting the points of $A$ onto a paraboloid in $\Rspace^3$, and then considering the lower surface of the convex hull of the $k$-fold averages, which project to the points in $A^{(k)}$.
Another way to construct $\Delaunay{k}{A}$ is from the dual order-$k$ Voronoi tessellation, as illustrated for $k=2$ in Figure~\ref{fig:Voronoi2}.
\begin{figure}[hbt]
  \centering
  \vspace{0.0in}
  \resizebox{!}{2.0in}{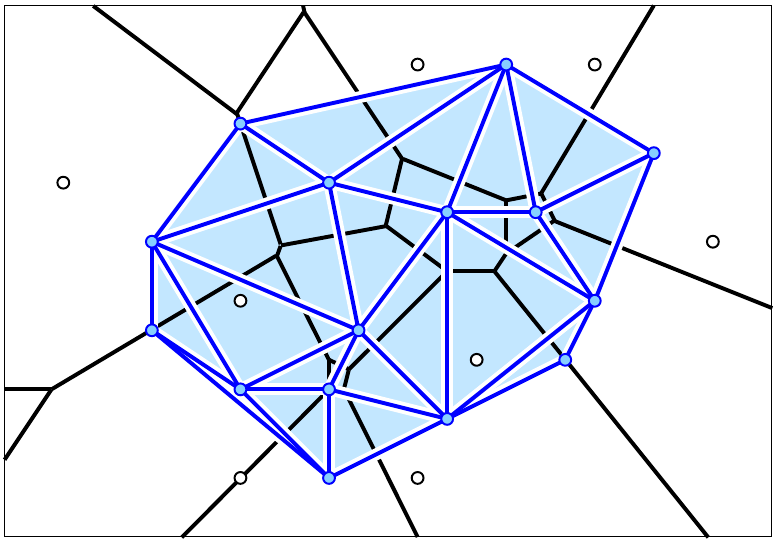}
  \caption{The (\emph{blue}) order-$2$ Delaunay triangulation drawn on top of the (\emph{black}) order-$2$ Voronoi tessellation of the set $A = \{a,b,\ldots,h\}$.
  Not all parts of the order-$2$ Voronoi tessellation are visible in the rectangular window.}
  \label{fig:Voronoi2}
\end{figure}

\smallskip
Note that for $k=1$, we get precisely the Delaunay triangulation of $A$, as all triangles are white and satisfy the empty circle criterion.
For $k = \card{A}-1$, we get the (scaled and centrally inverted copy of) the farthest-point Delaunay triangulation \cite{Epp92}.
Each of its triangles is black, and every point of $A$ is either a vertex or inside the circumcircle of the triangle.
Moreover, the aging function applies, and we have $\Delaunay{k+1}{A} = F(\Delaunay{k}{A})$ for every $1 \leq k < \card{A}-1$.

\subsection{Angles of Black and White Triangles}
\label{sec:3.2}

We now generalize the local angle property from order-$1$ to order-$k$.
For $2 \leq k \leq \card{A}-2$, we have black as well as white triangles.
Hence, there are three types of interior edges: those shared by two white triangles, two black triangles, and a white and a black triangle.
We have a different condition for each type.
\begin{definition}[Local Angle Property]
  \label{def:local_angle_property}
  Let $A \subseteq \Rspace^2$ be finite and generic.
  A level-$k$ hypertriangulation of $A$ has the \emph{local angle property} if
  \begin{itemize}
    \item \textsc{(ww)} for every edge shared by two white triangles, the sum of the two angles opposite the edge is at most $\pi$;
    \item \textsc{(bb)} for every edge shared by two black triangles, the sum of the two  angles opposite the edge is at least $\pi$;
    \item \textsc{(bw)} for every edge shared by a black triangle and a white triangle, the angle opposite the edge in the black triangle is bigger than the angle opposite the edge in the white triangle.
  \end{itemize}
\end{definition}
For $k=1$, there are no black triangles, so \textsc{(bb)} and \textsc{(bw)} are void.
Delaunay~\cite{Del34} proved that the local angle property characterizes the (closest-point) Delaunay triangulation among all (complete) triangulations of a finite point set, and this was used by Lawson~\cite{Law77} to construct the triangulation by repeated edge flipping.
Symmetrically, for $k = \card{A}-1$, there are no white triangles, so \textsc{(ww)} and \textsc{(bw)} are void.
Eppstein~\cite{Epp92} proved the local angle property for the (farthest-point) Delaunay triangulation, and the convergence of the flip-algorithm implies that it is the only (not necessarily complete) triangulation of the points that has this property.
The goal of this section is to extend these results to level-$k$ hypertriangulations.

\subsection{All Delaunay Triangulations Have the Local Angle Property}
\label{sec:3.3}

We prove that the Delaunay triangulations of any order have the local angle property.
This extends the results from $k = 1, \card{A}-1$ to any order between these limits.
\begin{theorem}[Order-$k$ Delaunay Triangulations have Local Angle Property]
  \label{thm:order-k_Delaunay_triangulations_have_lap}
  Let $A \subseteq \Rspace^2$ be finite and generic. 
  Then for every integer $1 \leq k \leq \card{A}-1$, the order-$k$ Delaunay triangulation of $A$ has the local angle property.
\end{theorem}
\begin{proof}
  Recall that white triangles of the order-$k$ Delaunay triangulation of $A$ have vertices $\Label{Y\!a}$, $\Label{Y\!b}$, $\Label{Y\!c}$, in which $Y \subseteq A$ with $\card{Y} = k-1$, such that all points of $Y$ are inside and all other points of $A$ are outside the circumcircle of $abc$.
  Similarly, its black triangles have vertices labeled $\Label{X\!ab}$, $\Label{X\!ac}$, $\Label{X\!bc}$, in which $X \subseteq A$ with $\card{X} = k-2$, such that all points of $X$ are inside and all other points of $A$ are outside this circumcircle. 
  We establish each of the three conditions separately.
  \begin{figure}[hbt]
    \centering
    \vspace{0.0in}
    \resizebox{!}{2.4in}{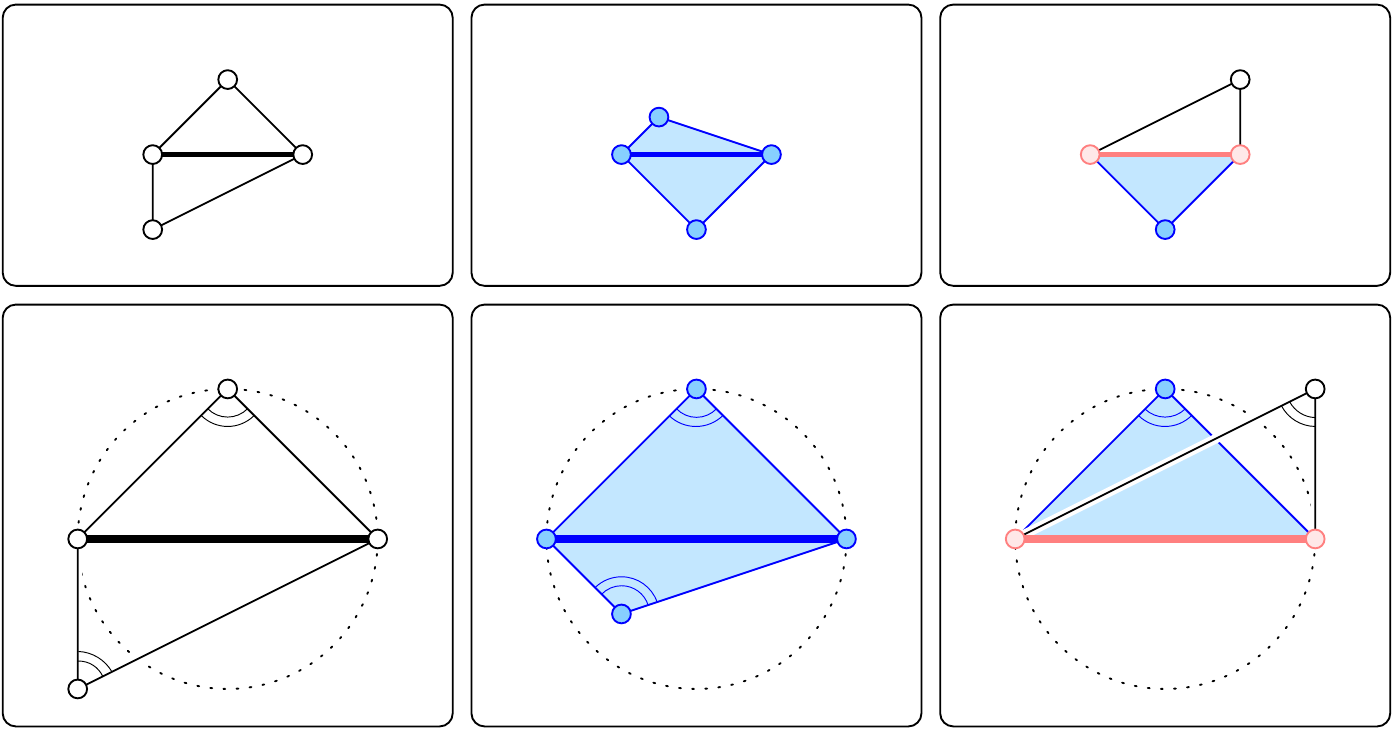}
      \caption{From \emph{left} to \emph{right:} an edge shared by two white triangles, two black triangles, a black triangle and a white triangle.
      \emph{Top row:} the adjacent triangles in the order-$k$ Delaunay triangulation. 
      The vertex labels encode the locations of the vertices as averages of the listed points.
      \emph{Bottom row:} the corresponding triangles spanned by the original points.}
      \label{fig:ww-bb-wb}
  \end{figure}

  \smallskip \noindent \textsc{(ww):}
  Let $\Label{Y\!a}, \Label{Y\!b}, \Label{Y\!c}$ and $\Label{Y\!b}, \Label{Y\!c}, \Label{Y\!d}$ be the vertices of two adjacent white triangles in the order-$k$ Delaunay triangulation of $A$, and note that the points of $Y$ lie inside and $d$ lies outside the circumcircle of $abc$; see the left panel of Figure~\ref{fig:ww-bb-wb}.
  The triangles $abc$ and $bcd$ are homothetic copies of these two white triangles, which implies that $a$ and $d$ lie on opposite sides of $bc$.
  Hence, $\measuredangle bac +\measuredangle bdc <\pi$, because $d$ is outside the circumcircle. 
  \textsc{(ww)} follows.

  \smallskip \noindent \textsc{(bb):} 
  Let $\Label{Z\!abc}$, $\Label{Z\!abd}$, $\Label{Z\!acd}$ and $\Label{Z\!abd}$, $\Label{Z\!acd}$, $\Label{Z\!bcd}$ be the vertices of adjacent black triangles in the order-$k$ Delaunay triangulation of $A$, and note that the points of $Z$ and $d$ lie inside the circumcircle of $abc$; see the middle panel of Figure~\ref{fig:ww-bb-wb}.
  The triangles $bcd$ and $abc$ are homothetic copies of these two black triangles, which implies that $a$ and $d$ are on opposite sides of $bc$.
  Hence, $\measuredangle bac +\measuredangle bdc >\pi$, because $d$ is inside the circumcircle. 
  \textsc{(bb)} follows.

  \smallskip \noindent \textsc{(bw):}
  Let $\Label{X\!ab}$, $\Label{X\!ac}$, $\Label{X\!bc}$ and $\Label{X\!ab}$, $\Label{X\!ac}$, $\Label{X\!ad}$ be the vertices of a black triangle and an adjacent white triangle in the order-$k$ Delaunay triangulation of $A$, and note that the points of $X$ lie inside while $d$ lies outside the circumcircle of $abc$; see the right panel of Figure~\ref{fig:ww-bb-wb}.
  The triangles $abc$ and $bcd$ are homothetic copies of the black and white triangles, with negative and positive homothety coefficients, respectively, which implies that $a$ and $d$ lie on the same side of $bc$.
  Thus, $\measuredangle bac >\measuredangle bdc $, because $d$ is outside the circumcircle. 
  \textsc{(bw)} follows.
\end{proof}
We conjecture that among all level-$k$ hypertriangulations that satisfy the local angle property, the order-$k$ Delaunay triangulation is the only one that maximizes the number of triangles.
For later reference, we refer to this as the \emph{Local Angle Conjecture} for hypertriangulations.
We discuss partial results for it in Section \ref{sec:5} and rigorously formulate it as Conjecture \ref{conj:B} in Section \ref{sec:6}.

\subsection{Constrained Delaunay Triangulations}
\label{sec:3.4}

Given a bounded polygonal region, $R$, it is always possible to find a triangulation, $P$, of its vertices (the endpoints of its edges) that contains all edges of the region.
Hence, every triangle of $P$ lies either completely inside or completely outside the region.
The \emph{restriction} of $P$ to $R$ consists of the triangles inside $R$, and we call this restriction a \emph{triangulation} of $R$.
For some choices of $P$, the restriction to $R$ looks locally like the Delaunay triangulation, namely when every edge that passes through the interior of $R$ satisfies \textsc{(ww)}.
It is not difficult to see that such choices of triangulations exist and that their restriction to $R$ is generically unique:
run Lawson's algorithm on an initial triangulation of $R$, flipping an interior edge whenever the sum of the two opposite angles exceeds $\pi$.
This is the \emph{constrained Delaunay triangulation} of $R$, as introduced in 1989 by Paul Chew \cite{Che89}, but see also \cite{CDS12,HNU99}.
A triangle $uvw$ belongs to this specific triangulation iff it is contained in $R$ and its circumcircle does not enclose any vertex that is visible from all points inside the triangle.
Here, a point $x$ is \emph{visible} from a point $y$, if the segment $xy$ lies inside $R$.
We state a weaker necessary condition for later reference.
\begin{lemma}[Triangles and Edges in Constrained Delaunay Triangulation]
  \label{lem:triangles_in_constrained_Delaunay_triangulation}
  Let $R$ be a bounded polygonal region in $\Rspace^2$, assume its vertex set is generic, and let $u, v, w$ be vertices of $R$.
  If the triangle $uvw$ is contained in $R$, and its circumcircle does not enclose any vertex of $R$, then $uvw$ is a triangle in the constrained Delaunay triangulation of $R$.
  Similarly, if the segment $uv$ is contained in $R$ but is not an edge of $R$, and it has a circumcircle that does not enclose any vertex of $R$, then $uv$ is an edge of the constrained Delaunay triangulation of $R$.
\end{lemma}

\smallskip
We use constrained Delaunay triangulations to decompose white regions in aged hypertriangulations.
To explain, let $P$ be a complete triangulation of a finite and generic set, $A \subseteq \Rspace^2$, let $x \in A$ be a vertex of this triangulation, call $\white{P,x} = \Star{P,x} \cap \conv{(A \setminus \{x\})}$ the \emph{white region} of $x$ in $P$, and let $P(x)$ be a triangulation of $\white{P,x}$.
Note that $\white{P,x} = \Star{P,x}$ if $x$ is an interior vertex, and $\white{P,x} \subsetneq \Star{P,x}$ if $x$ is a convex hull vertex.
In the special case in which $P$ is the order-$1$ Delaunay triangulation and $P(x)$ is the constrained Delaunay triangulation of $\white{P,x}$, for each $x \in A$, these sets contain all white triangles in the order-$2$ Delaunay triangulation, albeit the latter are only half the size. 
This can be seen from the following observation. If we consider the (order-1) Delaunay triangulation of $A\setminus\{x\}$, then it contains \emph{old} triangles from $P$ not incident to $x$ and \emph{new} triangles that retriangulate $\Star{P,x} \cap \conv{(A \setminus \{x\})}$. 
The circumcircle of each new triangle encloses $x$.
Thus, the new triangles are the white triangles from Definition~\ref{def:order-k_Delaunay_triangulation} and also the triangles from the constrained Delaunay triangulation of $\Star{P,x} \cap \conv{(A \setminus \{x\})}$.

\medskip
More generally, we use the constrained Delaunay triangulations of the white regions to disambiguate the aging function.
This is done extensively in the proofs of our main results in Section~\ref{sec:4}, where they are useful because constrained Delone triangulations optimize angle vectors, and in Section~\ref{sec:5}, where they are useful because they satisfy the local angle property.

\section{Optimality of the Sorted Angle Vector}
\label{sec:4}

In this section, we prove the first main result of this paper in an exhaustive case analysis.
With the exception of Section~\ref{sec:4.4}, we work only with complete level-$2$ hypertriangulations.
To aid the discussion, we begin by introducing convenient terminology and stating a few elementary lemmas.

\subsection{Triangulations and Angle Vectors}
\label{sec:4.1}

Let $A \subseteq \Rspace^2$ be a finite set of points,
and let $P$ be a complete triangulation of $A$, and write $P_2 = F(P)$ for the (complete) level-$2$ hypertriangulation whose white regions are decomposed by constrained Delaunay triangulations.
We prefer to work with the original points of $A$, rather than the midpoints of its pairs.
We therefore write $\Phi_2 = f(P)$ for the collection of triangles in $P$, together with the triangles in the constrained Delaunay triangulations of the $\white{P,x}$, with $x \in A$.
Consistent with the earlier convention, we call the triangles of $\Phi_2$ in $P$ \emph{black} and the other triangles of $\Phi_2$ \emph{white}.
Accordingly, we write $\Black{\Phi_2}$ for the black triangles in $\Phi_2$, and $\White{\Phi_2,x}$ for the white triangles in $\Phi_2$ that triangulate $\white{P,x}$.
There is a bijection between $\Phi_2$ and $P_2$ such that the corresponding triangles are similar (scaled by a factor $\frac{1}{2}$ and possibly inverted), so the triangles in $\Phi_2$ and $P_2$ define the same angles.
Letting $m$ be the number of triangles, we write
$
  \Vector{P_2}  =  \Vector{\Phi_2}  =  ( \varphi_1, \varphi_2, \ldots, \varphi_{3m} )
$
for the vector of angles, which we order such that $\varphi_i \leq \varphi_{i+1}$ for $1 \leq i \leq 3m-1$. 

\smallskip
Repeating the construction with another (maximal) triangulation $Q$ of $A$, we get another (complete) level-$2$ hypertriangulation of $m$ black and white triangles, $Q_2$, and another increasing angle vector, $\Vector{Q_2} = \Vector{\Psi_2} = (\psi_1, \psi_2,\ldots,\psi_{3m})$, in which $\Psi_2 = f(Q)$.
It is \emph{lexicographically larger} than the vector of $\Phi_2$, denoted $\Vector{\Phi_2} \prec \Vector{\Psi_2}$, if there exists an index $1 \leq p \leq m$ such that $\varphi_i = \psi_i$, for $1 \leq i \leq p-1$, and $\varphi_p < \psi_p$.
We write $\Vector{\Phi_2} \preceq \Vector{\Psi_2}$ to allow for the possibility of equal angle vectors.
This notation is useful because it is possible that two different triangulations, $P \neq Q$, have the same angle vector.
For example, if $A$ has only $4$ points and they are in convex position, then there are only two different triangulations of $A$, and the black triangles in the level-$2$ hypertriangulation of one are the white triangles in the level-$2$ hypertriangulations of the other, and vice versa.

\smallskip
According to Lemma~\ref{lem:complete_implies_maximal}, all maximal level-$2$ hypertriangulations have the same number of triangles.
We can therefore compare their angle vectors lexicographically as described.

\subsection{Elementary Lemmas}
\label{sec:4.2}

If $uvw$ is a triangle in $\White{\Phi_2,x}$, then it is not possible that $u$ lies inside $xvw$.
This is true independent of how we triangulate $\white{P,x}$:
\begin{lemma}[Star-convex Triangulation]
  \label{lem:star-convex_triangulation}
  Let $uvw$ be a triangle in $\White{\Phi_2,x}$.
  Then either $x$ is inside $uvw$ or $x, u, v, w$ are the vertices of a convex quadrangle.
  {\sloppy
  
  }
\end{lemma}
\begin{proof}
  Assume first that $x$ is an interior vertex, so $\conv{(A \setminus \{x\})} = \conv{A}$.
  Since $\white{P,x}$ is star-convex, with $x$ in its kernel, every half-line emanating from $x$ intersects the boundary of $\white{P,x}$ in exactly one point.
  Now suppose $u$ lies inside the triangle $xvw$, and consider the half-line emanating from $x$ that passes through $u$.
  Since $x$ lies in the interior of $\white{P,x}$, the half-line goes from inside to outside the region as it passes through $u$.
  But it also enters the triangle $uvw$, which lies inside $\white{P,x}$.
  This is a contradiction because entering and leaving $\Star{P,x}$ at the same time is impossible.

  Assume second that $x$ is a vertex of $\conv{A}$, so $\conv{(A \setminus \{x\})} \neq \conv{A}$.
  Since $uvw$ is a triangle in $\white{P,x}$, it is also a triangle in $\Star{P,x}$.
  Furthermore, $u, v, w$ are points on the boundary of $\Star{P,x}$, and every half-line emanating from $x$ that has a non-empty intersection with the interior of $\conv{A}$ intersects this boundary in exactly one point.
  Assuming $u$ lies inside $xvw$, we can now repeat the argument of the first case and get a contradiction because the half-line passing through $u$ both enters and leaves $\Star{P,x}$ when it passes through $u$.
\end{proof}

Every point $x \in A$ belongs to at least two edges in $P$.
However, if $x$ belongs to only two edges, then every line that crosses both edges necessarily separates $x$ from all points in $A \setminus \{x\}$.
We state and prove a generalization of this observation.
\begin{lemma}[Splitting a Triangulation]
  \label{lem:splitting_a_triangulation}
  Let $P$ be a triangulation of a finite set $A \subseteq \Rspace^2$, let $L$ be a line, and let $Q$ be the vertices and edges of $P$ that are disjoint of $L$.
  Then $Q$ consists of at most two connected components, one on each side of $L$.
\end{lemma}
\begin{proof}
  Assume without loss of generality that $L$ is horizontal, and let $A' \subseteq A$ contain all points strictly above $L$.
  The boundary of $\conv{A}$ is a closed convex curve, $\gamma$, and we write $\gamma' \subseteq \gamma$ for the vertices and edges strictly above $L$.
  Every point $a \in A'$ is either a vertex of $\gamma'$, or there is an edge $ab$ in $P$, with $b$ above $L$ and farther from $L$ than $a$.
  Hence, $ab \in Q$.
  We can therefore trace a path from $a$ that eventually reaches a vertex of $\gamma'$ in $Q$, which implies that the part of $Q$ strictly above $L$ is either empty or connected.
  Symmetrically, the part of $Q$ strictly below $L$ is either empty or connected, which implies the claim.
\end{proof}

By construction, the interior points of a black triangle, $abc \in P$, belong to $\Star{P,a}$, $\Star{P,b}$, $\Star{P,c}$ but not to the stars of any other vertices.
Hence, only the white triangles used in the triangulation of these three stars can possibly share interior points with $abc$.
If a white triangle shares one or two of the vertices with $abc$, then this further restricts the stars this white triangle may help triangulate.
\begin{lemma}[Shared Interior Points]
  \label{lem:shared_interior_points}
  Let $P$ be a triangulation of a finite set $A \subseteq \Rspace^2$, let $abc$ be a black triangle and $uvw$ a white triangle in $\Phi_2 = f(P)$, and suppose that $abc$ and $uvw$ share interior points.
  \begin{enumerate}
    \item[{\rm (1)}] If $u=a$ and $v=b$, then $uvw \in \White{\Phi_2, c}$.
    \item[{\rm (2)}] If $v=b$ is the only shared vertex between $abc$ and $uvw$, then $uw$ cannot cross both $ab$ and $bc$.
    \item[{\rm (3)}] If $v = b$ and $uw$ crosses $bc$, then $uvw \in \White{\Phi_2, c}$.
    \item[{\rm (4)}] $uvw \in \White{\Phi_2, x}$ for only one point $x \in A$.
  \end{enumerate}
\end{lemma}
\begin{proof}
  (1) is immediate because $c$ is the only vertex of $abc$ that is not also a vertex of $uvw$.
  
  To see (2), assume that $uw$ crosses $ab$ and also $bc$.
  Then $uvw$ shares interior points with three black triangles in $\Phi_2$, namely $abc$ and the neighboring triangles that share $ab$ and $bc$ with $abc$.
  The only common vertex of the three black triangles is $b$, so $uvw \in \White{\Phi_2, b}$, but this is impossible because $b=v$.

  To see (3), note that $uvw$ shares interior points with two black triangles: $bac$ and the black triangle on the other side of $bc$.
  Hence, $uvw$ is contained in $\Star{P,b}$ or $\Star{P,c}$.
  Since $b = v$, the only remaining choice is $uvw \in \White{\Phi_2, c}$.

  To see (4), consider first the case that $uvw$ shares interior points with only two black triangles, $abc$ and $bcd$.
  Then one of its edges, say $uv$ crosses $bc$, so $u=a$ and $w=d$.
  But $v$ cannot lie in the interior of the two black triangles or its edges, so $v=b$.
  Then $c$ is the only remaining point such that $uvw \in \White{\Phi_2, c}$.
  If $uvw$ shares interior points with three or more black triangles, then the black triangles share only one common vertex, $x$, hence $uvw \in \White{\Phi_2, x}$.
\end{proof}

\subsection{Global Optimality}
\label{sec:4.3}

The first main result of this paper asserts that Sibson's theorem on increasing angle vectors extends from order-$1$ to order-$2$ Delaunay triangulations.
We first illustrate our approach by giving a new proof of Sibson's angle vector optimality for the order-1 Delaunay triangulation. 
Note that we establish the non-strict optimality while Sibson's theorem gives the strict one in the generic case.
\begin{theorem}[Sibson's Theorem \cite{Sib78}]
  \label{thm:sibson}
  Let $A \subseteq \Rspace^2$ be finite and generic, $P$ a complete triangulation of $A$, and $D = \Delaunay{}{A}$.
  Then $\Vector{P} \preceq \Vector{D}$.
\end{theorem}

\begin{proof}
  The genericity of $A$ implies $\measuredangle xay \neq \measuredangle xby$ whenever points $a, b \in A$ lie on the same side of the line that passes through $x, y \in A$.
  However, there may be two or more triplets of points in $A$ that define the same angle.
  It will be convenient to have distinct angles, so we first apply a perturbation that preserves the order of unequal angles while making equal angles different. 
  The relation for the perturbed points implies the same but possibly non-strict relation for the original points, since undoing the perturbation does not change the order of any two angles.
  So assume that the angles defined by the points in $A$ are distinct, and to derive a contradiction, assume $\Vector{D} \prec \Vector{P}$.
  More specifically, we write $\alpha_1 < \alpha_2 < \ldots < \alpha_{3m}$ and $\varphi_1 < \varphi_2 < \ldots < \varphi_{3m}$ for the angles of $D$ and $P$, respectively, and we assume $\alpha_i = \varphi_i$, for $1 \leq i \leq p-1$, and $\alpha_p < \varphi_p$, for some $1 \leq p \leq 3m$.
  In other words, $p$ is the first index at which the two angle vectors differ, and the $p$-th angle of $D$ is smaller than the $p$-th angle of $P$.
  Write $\alpha = \alpha_p$ and let $bac \in D$ be the triangle with $\alpha = \measuredangle bac$.
  By the assumption of distinct angles, $bac \not\in P$. To simplify the discussion, we assume that the line, $L$, that passes through $b$ and $c$ is horizontal and that $a$ lies above $L$.
  Since $P$ is a complete triangulation of $A$, every point of $A$ is a vertex of $P$. We consider two cases for the triangles of $P$ that cover the upper side of $bc$\footnote{We use the phrase ``to cover the upper side of $bc$'' informally, in the sense that we consider triangles that intersect or contain $bc$ and some parts of an open neighborhood of $bc$ above $bc$. A more formal explanation is given in the proof of Theorem \ref{thm:angle_vector_optimality}.}.

  \smallskip
  If there is only one such triangle, $bdc$, then $d\neq a$ as $P$ does not contain the triangle $bac$. Since $D$ is the Delaunay triangulation of $A$, $d$ lies outside the circumcircle of $bac$ and therefore $\measuredangle bdc < \alpha$. This implies that $bdc$ must be a triangle of $D$ as the angle vectors of $D$ and $P$ coincide at angles less than $\alpha$. 
  This is a contradiction as both $bac$ and $bdc$ cannot belong to $D$.

  \smallskip
  If there are at least two triangles of $P$ covering the upper side of $bc$, then we consider the triangle $bxy$ that shares $b$ with $bac$.
  Here $x$ is above $L$ and $y$ is below $L$. Since $c$ is a vertex of $P$, $xy$ intersects $bc$ and $\measuredangle bxy < \measuredangle bxc < \alpha$. 
  As in the first case, this implies that $bxy$ is a triangle in $D$, which leads to a contradiction as the interiors of the triangles $bxy$ and $bac$ overlap.
\end{proof}

We now use a similar approach to establish the angle vector optimality for the order-2 Delaunay triangulation. 

\begin{theorem}[Angle Vector Optimality]
  \label{thm:angle_vector_optimality}
  Let $A \subseteq \Rspace^2$ be finite and generic, $P_2$ any complete level-2 hypertriangulation of $A$, and $\Delta_2 = f(\Delaunay{}{A})$.
  Then $\Vector{P_2} \preceq \Vector{\Delta_2}$.
\end{theorem}
\begin{proof}
  Let $P$ be the complete triangulation of $A$ such that $P_2=F(P)$. Note, however, that we do not assume that the white triangles of $P_2$ define constrained Delaunay triangulations of the white regions. 
  Let $\Phi_2=F(P)$ be the designated aging of $P$ so $\Phi_2$ shares its black triangles with $P_2$, and the white triangles of $\Phi_2$ form constrained Delaunay triangulation of the white regions. 
  For each white region of $P_2$ and $\Phi_2$, we apply Lawson's edge-flip algorithm to turn the white triangles of $P_2$ into the white triangles of $\Phi_2$; see \cite[Section~2.10.3]{CDS12} for more details about this algorithm.
  Since every flip lexicographically improves the angle vector, we get $\Vector{P_2}\preceq \Vector{\Phi_2}$, so it remains to show that $\Vector{\Phi_2}\preceq \Vector{\Delta_2}$.

  \smallskip
  Write $D = \Delaunay{}{A}$, so $\Delta_2 = f(D)$.
  The genericity of $A$ implies that $D$ and $\Delta_2$ are unique, but there may be two or more triplets of points that define the same angle.
  As in the proof of Theorem \ref{thm:sibson}, it will be convenient to have distinct angles, so we apply a perturbation and assume that the angles defined by the points in $A$ are distinct.
  To derive a contradiction, assume $\Vector{\Delta_2} \prec \Vector{\Phi_2}$. 
  More specifically, we write $\alpha_1 < \alpha_2 < \ldots < \alpha_{3m}$ and $\varphi_1 < \varphi_2 < \ldots < \varphi_{3m}$ for the angles of $\Delta_2$ and $\Phi_2$, respectively, and we assume $\alpha_i = \varphi_i$, for $1 \leq i \leq p-1$, and $\alpha_p < \varphi_p$, for some $1 \leq p \leq 3m$.
  In other words, $p$ is the first index at which the two angle vectors differ, and the $p$-th angle of $\Delta_2$ is smaller than the $p$-th angle of $\Phi_2$.
  Write $\alpha = \alpha_p$ and let $bac \in \Delta_2$ be the triangle with $\alpha = \measuredangle bac$.
  By the assumption of distinct angles, $bac \not\in \Phi_2$.
  To simplify the discussion of the various cases, we assume without loss of generality that
  \begin{itemize}
    \item the line, $L$, that passes through $b$ and $c$ is horizontal;
    \item the triangle $bac$, and therefore the vertex $a$, lie above $L$;
  \end{itemize}
  see Figures~\ref{fig:Case1} and \ref{fig:Case2}.
  We first consider the case in which $bac$ is a black triangle.
  There are three subcases, and in each we get a contradiction by constructing two triangles that share interior points.
  Note that two white triangles may share interior points, but not if they triangulate the same star.
  \begin{figure}[hbt]
    \centering \vspace{0.1in}
    \resizebox{!}{1.68in}{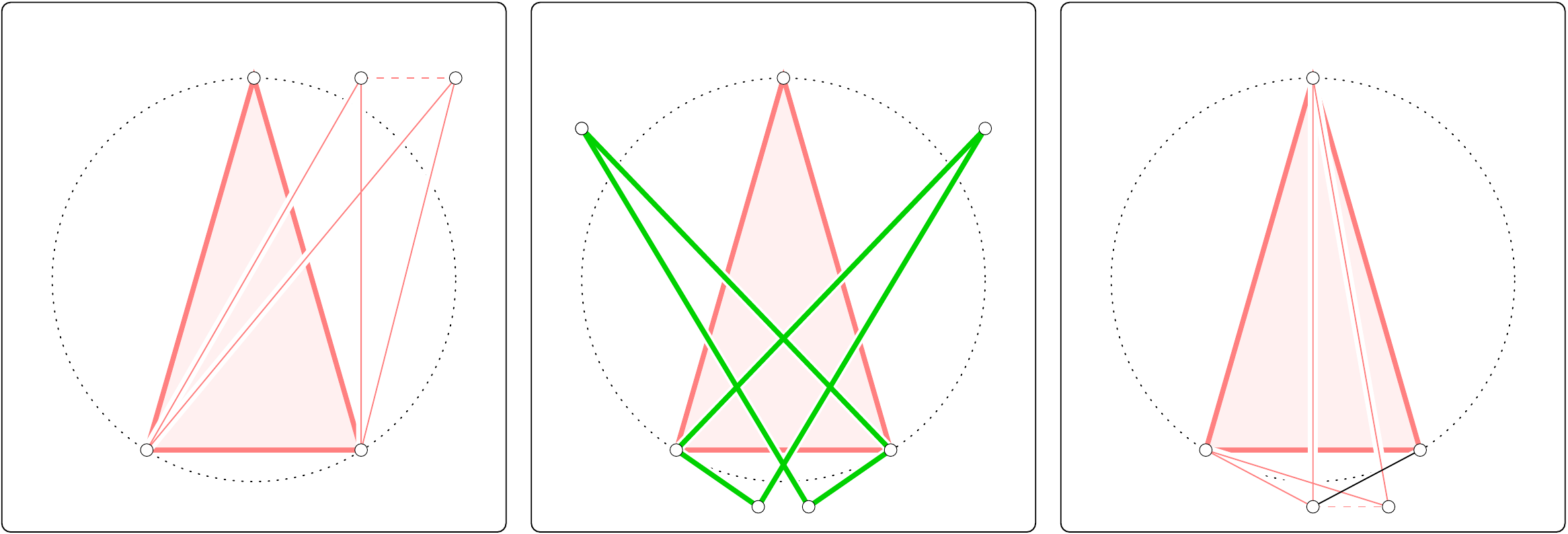}
    \caption{Edges of black and white triangles are \emph{bold} and \emph{fine}, respectively, and edges of triangles in $\Delta_2$ and $\Phi_2$ are \emph{pink} and \emph{green}, respectively.
    \emph{Left:} two overlapping triangles in $\White{\Delta_2,a}$ constructed in Case 1.1.
    \emph{Middle:} two crossing edges of black triangles in $\Phi_2$ constructed in Case 1.2.1.
    \emph{Right:} two overlapping triangles in $\White{\Delta_2,c}$ constructed in Case 1.2.2.}
    \label{fig:Case1}
  \end{figure}

  \noindent \textsc{Case 1:}
  {\bf $bac$ is a black triangle in $\Delta_2$.}
  By definition of $D = \Delaunay{}{A}$, $bac$ does not contain a point of $A$ in its interior, and if $x \in A \setminus \{a\}$ lies above $L$, then the angle $\measuredangle bxc$ is strictly smaller than $\alpha$.
  We say a collection of triangles \emph{covers the upper side} of the edge $bc$ if every interior point of $bc$ has an open neighborhood whose intersection with the closed half-plane above $L$ is contained in the union of these triangles.
  The black triangles in $\Phi_2$ cover the entire convex hull of $A$ and therefore also the upper side of $bc$.
  It is possible that a single black triangle in $\Phi_2$ suffices for this purpose, and this is our first subcase.

  \smallskip \noindent \textsc{Case 1.1:}
  {\bf the upper side of $bc$ is covered by a single triangle, $bxc \in \Black{\Phi_2}$,} as in Figure~\ref{fig:Case1} on the left.
  Since $\measuredangle bxc < \alpha$, $bxc$ must be a white triangle in $\Delta_2$.
  Specifically, since $a$ and $x$ are both above $L$, and $a$ lies inside the circumcircle of $bxc$, we have $bxc \in \White{\Delta_2, a}$.

  To get a contradiction, we construct a second such white triangle.
  Since there are at least two points of $A$ above $L$, Lemma~\ref{lem:splitting_a_triangulation} implies that $P$ contains an edge connecting $x$ to another point, $x' \neq x$, above $L$.
  Hence, $\white{P,x}$ has a non-empty overlap with the open half-plane above $L$.
  Since $bc$ belongs to the boundary of $\white{P,x}$, there is a triangle $bx'c$ in $\White{\Phi_2,x}$.
  We have $x' \neq x$ by construction, and $x' \neq a$ because this would imply that $\measuredangle bx'c = \alpha$ is an angle in $\Vector{\Phi_2}$, which we assumed it is not.
  Since $x'$ lies outside the circumcircle of $bac$, we have $\measuredangle bx'c < \alpha$, so $bx'c \in \White{\Delta_2,a}$.
  But $bxc$ and $bx'c$ share interior points, which is a contradiction. 

  \smallskip \noindent \textsc{Case 1.2:}
  {\bf to cover the upper side of $bc$ requires two or more triangles in $\Black{\Phi_2}$}, as in Figure~\ref{fig:Case1} in the middle and on the right.
  Among these triangles, let $bxy$ and $cx'y'$ be the ones that share the vertices $b$ and $c$ with $bac$.
  Assuming $x, x'$ lie above $L$ and $y, y'$ lie below $L$, we have $\measuredangle bxy < \alpha$ and $\measuredangle cx'y' < \alpha$, which implies $bxy, cx'y' \in \Delta_2$.
  The two triangles share interior points with $bac$, so they cannot be black and are therefore white in $\Delta_2$.

  \smallskip \noindent \textsc{Case 1.2.1:}
  {\bf at least one of $x, x'$ differs from $a$.}
  Assume $x \neq a$.
  Since $xy$ crosses $bc$, it must cross another edge of $bac$, which by Lemma~\ref{lem:shared_interior_points} (2) can only be $ac$.
  If $x' = a$, then $x'c = ac$, and if $x' \neq a$, then $x'y'$ crosses $ab$ and $bc$, again by Lemma~\ref{lem:shared_interior_points} (2).
  In either case, $bxy$ and $cx'y'$ share interior points inside triangle $abc$, which contradicts $bxy, cx'y' \in \Black{\Phi_2}$.

  \smallskip \noindent \textsc{Case 1.2.2:}
  {\bf both $x$ and $x'$ are equal to $a$.}
  Then $bay, cay' \in \Black{\Phi_2}$.
  Since $\measuredangle bay < \alpha$ and $\measuredangle cay' < \alpha$, both are white triangles in $\Delta_2$. By Lemma~\ref{lem:shared_interior_points} (1), $bay \in \White{\Delta_2, c}$ and $cay' \in \White{\Delta_2, b}$, which implies that $cy$ and $by'$ are edges in $\Delaunay{}{A}$.
  If $y\neq y'$, then there are three possible choices for the points $b$, $c$, $y$, $y'$.
  First, they form a convex quadrangle, $byy'c$, with the points ordered as they are seen from $a$.
  The vertices of the quadrangle must be ordered that way as both triangles, $bay$ and $cay'$, cover parts of the upper side of $bc$ and cannot overlap since they both belong to $\Black{\Phi_2}$.
  But then $by'$ and $cy$ cross, which contradicts that they both belong to $\Delaunay{}{A}$.
  Second, $y$ lies inside $bcy'$.
  Since $cay'\in \White{\Delta_2,b}$, the circumcircle of $cay'$ encloses $b$ and therefore $y$, which is one point too many for a white triangle in $\Delta_2$.
  Third, $y'$ lies inside $bcy$, but this is symmetric to the second choice.
  Since we get a contradiction for all three choices, we conclude that $y=y'$.

  To get a contradiction, we use Lemma~\ref{lem:splitting_a_triangulation} to construct yet another triangle $baz \in \White{\Delta_2,c}$.
  Specifically, we let $L$ be the line that passes through $a$ and $b$, and rotate the picture so $L$ is horizontal and $c, y$ lie above $L$.
  Hence, there is a point $z$ above $L$ such that $yz$ is an edge in $P$ and $baz \in \White{\Phi_2,y}$.
  We have $z \neq y$ by construction, and $z \neq c$ by assumption on angle $\alpha$.
  Since $ba$ and $ac$ are both edges in the boundary of $\Star{P, y}$, $za$ crosses $bc$, so $\measuredangle baz < \alpha$, which implies that $baz$ is a white triangle in $\Delta_2$, and by Lemma~\ref{lem:shared_interior_points} (1), $baz \in \White{\Delta_2, c}$.
  But $bay$ and $baz$ share interior points, which is a contradiction.
  This concludes the proof of the first case.

  \begin{figure}[hbt]
    \centering \vspace{0.1in}
    \resizebox{!}{1.60in}{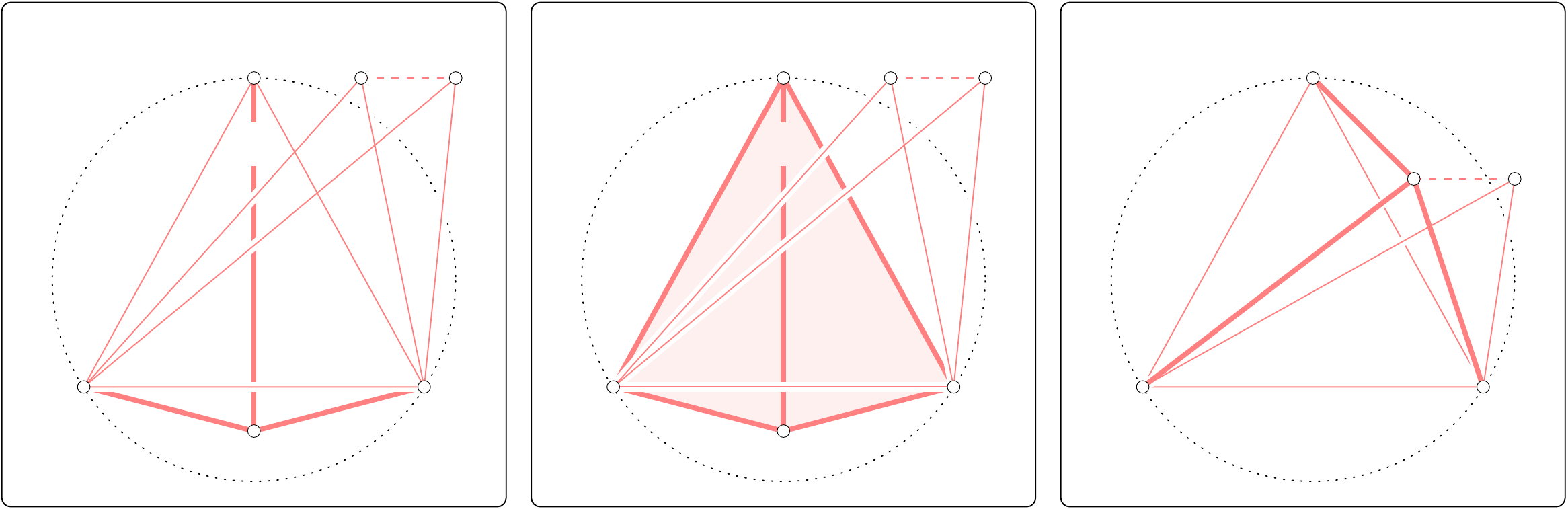}
    \caption{As before, we draw edges of black and white triangles \emph{bold} and \emph{fine}, respectively.
    To simplify, we show only edges of triangles in $\Delta_2$.
    \emph{Left:} two overlapping triangles in $\White{\Delta_2, a}$ constructed in Case~2.1.1.
    \emph{Middle:} similar two overlapping triangles in $\White{\Delta_2, a}$ constructed in a chain of deductions in Case~2.1.2.
    \emph{Right:} a white triangle whose circumcircle encloses two points constructed in Case~2.2.}
    \label{fig:Case2}
  \end{figure}

  \medskip \noindent \textsc{Case 2:}
  {\bf $bac$ is a white triangle in $\Delta_2$.}
  Let $d$ be the point such that $bac \in \White{\Delta_2,d}$.
  Then $da$, $db$, $dc$ are edges of black triangles in $\Delta_2$.
  We distinguish between the cases in which $d$ lies below and above $L$.

  \smallskip \noindent \textsc{Case 2.1:}
  {\bf $d$ lies below $L$;} see the left and middle panels of Figure~\ref{fig:Case2}.
  Then $\measuredangle bxc < \measuredangle bac$ for all $x \in A$ above $L$, and $\measuredangle byc < \measuredangle bdc$ for all $y \in A$ below $L$.
  Indeed, since $bac \in \White{\Delta_2,d}$, the circumcircle of $bac$ contains only $d$ inside, because only such white triangles are included in the order-2 Delaunay triangulation $\Delta_2$; see Definition \ref{def:order-k_Delaunay_triangulation}. 
  This implies the inequality for every $x\in A$ above $bc$.
  For every $y\in A$ below $bc$, we have $\measuredangle bac +\measuredangle byc <\pi$, since $y$ is outside the circumcircle $abc$, while $\measuredangle bac +\measuredangle bdc >\pi$, since $d$ is inside the circumcircle $abc$. 
  This implies $\measuredangle byc < \measuredangle bdc$.

  Similar to Case~1.1, we distinguish between the upper side of $bc$ being covered by one or requiring two or more black triangles in $\Phi_2$.
  In both cases, we derive a contradiction by constructing triangles in $\White{\Delta_2, a}$ that share interior points.

  \smallskip \noindent \textsc{Case 2.1.1:}
  {\bf the upper side of $bc$ is covered by a single triangle, $bxc \in \Black{\Phi_2}$}; see the left panel of Figure~\ref{fig:Case2}.
  Then $\measuredangle bxc < \alpha$, so $bxc$ is a triangle in $\Delta_2$, and since $a$ lies inside its circumcircle, we have $bxc \in \White{\Delta_2, a}$.
  Using Lemma~\ref{lem:splitting_a_triangulation}, we find a point $x'$ above $L$ such that $xx'$ is an edge in $P$ and $bx'c$ is a triangle in $\White{\Phi_2, x}$.
  We have $x' \neq x$ by construction, and $x' \neq a$, else $\measuredangle bx'c = \alpha$ would be an angle in $\Vector{\Phi_2}$.
  Again $\measuredangle bx'c < \alpha$, so $bx'c \in \White{\Delta_2, a}$.
  This is a contradiction because $bxc$ and $bx'c$ share interior points.

  \smallskip \noindent \textsc{Case 2.1.2:}
  {\bf to cover the upper side of $bc$ requires at least two triangles in $\Black{\Phi_2}$}.
  Among these triangles, let $bxy$ and $cx'y'$ be the ones that share $b$ and $c$ with $bac$, respectively, and assume that $x, x'$ are above $L$ and $y, y'$ are below $L$.
  We first prove that $d$ is connected to $b$ and $c$ by edges of black triangles in $\Phi_2$, and thereafter derive a contradiction by constructing two triangles in $\White{\Delta_2, a}$ that share interior points.

  \smallskip \noindent
  \emph{Claim:} $bd$ and $cd$ are edges of triangles in $\Black{\Phi_2}$.

  \smallskip \noindent
  \emph{Proof.}
    To derive a contradiction, assume the claim is false and $bd$ is not edge of any black triangle in $\Phi_2$.
    Hence $y \neq d$.
    Since $\measuredangle bxy < \alpha$, $bxy$ is also in $\Delta_2$.
    It shares interior points with the star of $d$ without having $d$ as a vertex, which implies that $bxy$ must be white in $\Delta_2$.

    Consider $bdc$, which is not necessarily a triangle in $\Delta_2$ or $\Phi_2$.
    However, since $d$ is the only point inside the circumcircle of $bac$, there is no point of $A$ inside $bdc$.
    Since $xy$ crosses $bc$, it must cross either $bd$ or $cd$.
    Assuming $xy$ crosses $bd$, $bxy$ shares interior points with the two black triangles with common edge $bd$ in $\Delta_2$, so $bxy \in \White{\Delta_2, d}$ by Lemma~\ref{lem:shared_interior_points} (3).
    This is not possible since $bxy$ and $bac$ share interior points.
    Thus, $xy$ crosses $cd$.
    Since $bxy \in \Black{\Phi_2}$, this implies that $cd$ cannot be edge of any black triangle in $\Phi_2$.
    Hence $y' \neq d$, so we can use the symmetric argument to conclude that $x'y'$ crosses $bd$.
    But this is a contradiction since in this case $bxy$ and $cx'y'$ share interior points inside the triangle $bcd$; see the middle panel of Figure~\ref{fig:Case1} where the situation is similar.
    This completes the proof of the claim.
    
  \smallskip
  Since $bd$ and $cd$ are edges of triangles in $\Black{\Phi_2}$, we have $y = y' = d$.
  Consider $\Star{P,d}$, which contains $b$ and $c$ on its boundary.
  The black triangles in $\Phi_2$ that cover the upper side of $bc$ all share $d$ as a vertex, which implies that $bc$ lies inside this star.
  Indeed, by Lemma~\ref{lem:triangles_in_constrained_Delaunay_triangulation}, it is an edge of a triangle in $\White{\Phi_2, d}$.
  Indeed, among the points of $A$, the circumcircle of $bac$ contains only $d$ inside, but $d$ is not a vertex of the polygonal region $\white{\Phi_2, d}$.
  Thus, there exists a triangle $bzc \in \White{\Phi_2, d}$ with $z$ above $L$.
  We have $z \neq a$ by assumption on $\alpha$, so $\measuredangle bzc < \alpha$, which implies that $bzc$ is also a white triangle in $\Delta_2$, and since its circumcircle encloses $a$, $bzc \in \White{\Delta_2, a}$.

  To construct a second such white triangle, note that this implies that $ab$ and $ac$ are edges of triangles in $\Black{\Delta_2}$.
  As illustrated in the middle panel of Figure~\ref{fig:Case2}, all of $ab, ac, ad, bd, cd$ are edges of black triangles in $\Delta_2$, so $bad, cad \in \Black{\Delta_2}$.
  Hence, $bd$ and $cd$ are edges in the boundary of $\Star{D,a}$, and since $bzc \in \White{\Delta_2, a}$, we also have $bdc \in \White{\Delta_2, a}$.
  The angle at $b$ satisfies $\measuredangle dbc < \measuredangle dac < \alpha$ because $a$ lies inside the circumcircle of $dbc$, and since $dbc$ is a triangle in $\Delta_2$, it must therefore also be a triangle in $\Phi_2$.
  It cannot be in $\Black{\Phi_2}$ because the upper side of $bc$ requires at least two black triangles of $\Phi_2$ to be covered, by assumption.
  Hence, $dbc$ is white in $\Phi_2$.
  It shares interior points with the two black triangles with common edge $dz$ in $\Phi_2$, so $dbc \in \White{\Phi_2, z}$, by Lemma~\ref{lem:shared_interior_points} (3).

  Finally consider $\White{\Phi_2, z}$.
  It contains $bdc$ and, by Lemma~\ref{lem:splitting_a_triangulation}, it covers the upper side of $bc$.
  Hence, there is a triangle $bz'c \in \White{\Phi_2, z}$ with $z'$ above $L$.
  We have $z' \neq z$ by construction, and $z' \neq a$ by assumption on $\alpha$.
  Again, $\measuredangle bz'c < \alpha$, so $bz'c \in \Delta_2$, and since its circumcircle encloses $a$, we have $bz'c \in \White{\Delta_2, a}$.
  But this is a contradiction because $bzc$ and $bz'c$ share interior points.

  \smallskip \noindent \textsc{Case 2.2:}
  {\bf $d$ lies above $L$;} see the right panel of Figure~\ref{fig:Case2}.
  Similar to Case 2.1.2, we begin by proving that $d$ is connected to $b$ and $c$ by edges of black triangles in $\Phi_2$.

  \smallskip \noindent
  \emph{Claim:} $bd$ and $cd$ are edges of triangles in $\Black{\Phi_2}$.

  \smallskip \noindent
  \emph{Proof.} To derive a contradiction, assume the claim is false and $bd$ is not edge of any black triangle in $\Phi_2$.
  Among the one or more black triangles needed to cover the upper side of $bc$, let $bxy \in \Black{\Phi_2}$ be the triangle that shares $b$ with $bac$.
  Letting $x$ be the vertex above $L$, we have $x \neq d$ by assumption.
  If $bxy$ covers the upper side of $bc$ by itself, then $y = c$, and otherwise, $y$ lies below $L$.
  In either case, $\measuredangle bxy < \alpha$, so $bxy$ is also a triangle in $\Delta_2$.
  It cannot be black because it shares interior points with $\Star{D, d}$ without having $d$ as a vertex, so $bxy$ is a white triangle in $\Delta_2$.
  But this implies $y \neq c$.
  Indeed, if $y = c$, then either $bxy = bac$, which contradicts the assumption on $\alpha$, or the circumcircle of $bxy$ encloses $a$ as well as $d$, which is one point too many for a white triangle in $\Delta_2$.

  So $y$ is below $L$.
  Note that the circumcircle of $bac$ encloses $d$ and therefore $bdc$, and since $x$ lies on or outside this circle, it cannot lie inside $bdc$.
  Since $xy$ crosses $bc$, it thus must cross another edge of this triangle, either $bd$ or $cd$.
  Assuming $xy$ crosses $bd$, which is common to two black triangles in $\Delta_2$, we get $bxy \in \White{\Delta_2, d}$ from Lemma~\ref{lem:shared_interior_points} (3).
  But $bxy$ and $bac \in \White{\Delta_2, d}$ share interior points, which is a contradiction.
  Hence, $xy$ crosses $bc$ and $cd$, so $cd$ cannot be an edge of a black triangle in $\Phi_2$.

  Let now $cx'y'$ be among the triangles in $\Black{\Phi_2}$ needed to cover the upper side of $bc$ that shares $c$ with $bac$.
  By a symmetric argument, we conclude that $x'y'$ crosses $bc$ and $bd$.
  But this is a contradiction because $bxy$ and $cx'y'$ share interior points inside the triangle $bcd$; see again the middle panel of Figure~\ref{fig:Case1} but substitute $d$ for $a$.
  This completes the proof of the claim.

  \smallskip
  Hence, $bd$ and $cd$ are edges of black triangles in $\Phi_2$.
  This implies that $b$ and $c$ are points in the boundary of $\Star{P, d}$.
  As argued above, there are no points of $A$ inside $bdc$, so $\Star{P, d}$ covers the upper side of $bc$.
  There is a circle that passes through $b$ and $c$ and encloses $d$ but no other points of $A$, so by Lemma~\ref{lem:triangles_in_constrained_Delaunay_triangulation}, $bc$ is an edge of a triangle in $\White{\Phi_2, d}$.
  Let $z$ be the point above $L$ such that $bzc \in \White{\Phi_2, d}$.
  We have $z \neq d$ by construction, and $z \neq a$ by assumption on $\alpha$.
  Hence, $\measuredangle bzc < \alpha$, which implies that $bzc$ is also a triangle in $\Delta_2$.
  However, the circumcircle of $bzc$ encloses $a$ and $d$, which is one too many for a white triangle in $\Delta_2$.
  This furnishes the final contradiction and completes the proof of the theorem.
\end{proof}

\subsection{Counterexamples}
\label{sec:4.4}

Can Theorem~\ref{thm:angle_vector_optimality} be extended or strengthened?
In this subsection, we present examples that contradict the extension to order beyond $2$ and the strengthening to order-$2$ hypertriangulations obtained from possibly incomplete triangulations.

\begin{figure}[hbt]
  \centering \vspace{0.1in}
  \resizebox{!}{0.93in}{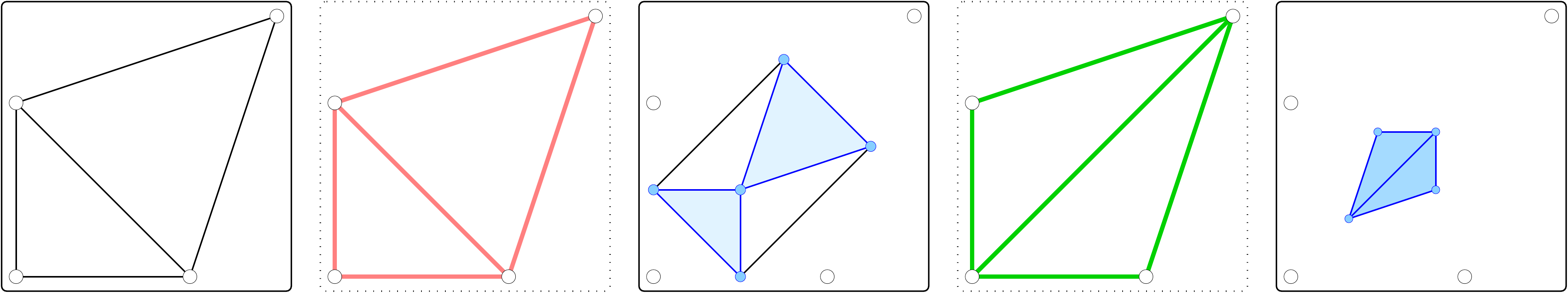}
  \caption{From \emph{left} to \emph{right}: the order-$1$, order-$2$, and order-$3$ Delaunay triangulations of four points, interleaved with the two possible triangulations of these points.}
  \label{fig:convex4}
\end{figure}
\medskip \noindent \textbf{Order beyond 2.}
Four points in convex position permit only two triangulations: $D = \Delaunay{}{A}$, and $P$, which consists of the other two triangles spanned by the four points.
As illustrated in Figure~\ref{fig:convex4}, $\Delaunay{2}{A}$ consists of shrunken and possibly inverted copies of all four triangles, and $\Delaunay{3}{A}$ consists of shrunken and inverted copies of the two triangles in $P$.
Assuming $A$ is generic, Sibson's theorem implies $\Vector{P} \prec \Vector{D}$.
There are two level-$3$ hypertriangulations: the order-$3$ Delaunay triangulation, with $\Vector{\Delaunay{3}{A}} = \Vector{P}$, and another, with $\Vector{P_3} = \Vector{D}$.
Hence, $\Vector{\Delaunay{3}{A}} \prec \Vector{P_3}$.
In words, the vector inequality asserted in Theorem~\ref{thm:angle_vector_optimality} for order-$2$ Delaunay triangulations does not even extend to order $3$.

\smallskip
Compare this with Eppstein's theorem \cite{Epp92}, which asserts that for $n$ points in convex position in $\Rspace^2$, the order-$(n-1)$ Delaunay triangulation lexicographically minimizes the increasing angle vector.
For $n = 4$ and points in convex position, the above conclusion is a consequence of this theorem.

\medskip \noindent \textbf{Incomplete hypertriangulations.}
Theorem~\ref{thm:angle_vector_optimality} compares the order-$2$ Delaunay triangulation with all \emph{complete} level-$2$ hypertriangulations, each aged from a triangulation that contains each point in $A$ as a vertex.
Enlarging this collection to possibly incomplete level-$2$ hypertriangulations is problematic since they do not necessarily have the same number of angles as $\Delaunay{2}{A}$.
We can still compare the smallest angles, but there are counterexamples.
Indeed, Figure~\ref{fig:min_counterexample} shows a set of nine points whose order-$2$ Delaunay triangulation does not maximize the minimum angle if incomplete level-$2$ hypertriangulations participate in the competition.
We note that for these particular nine points, the angle vectors of $\Delaunay{2}{A}$ and the displayed level-$2$ hypertriangulation have the same length.
This implies that the requirement of \emph{completeness} cannot be weakened to \emph{maximality}, which is equivalent to having the same number of triangles.
\begin{figure}[hbt]
  \centering \vspace{0.0in}
  \resizebox{!}{3.6in}{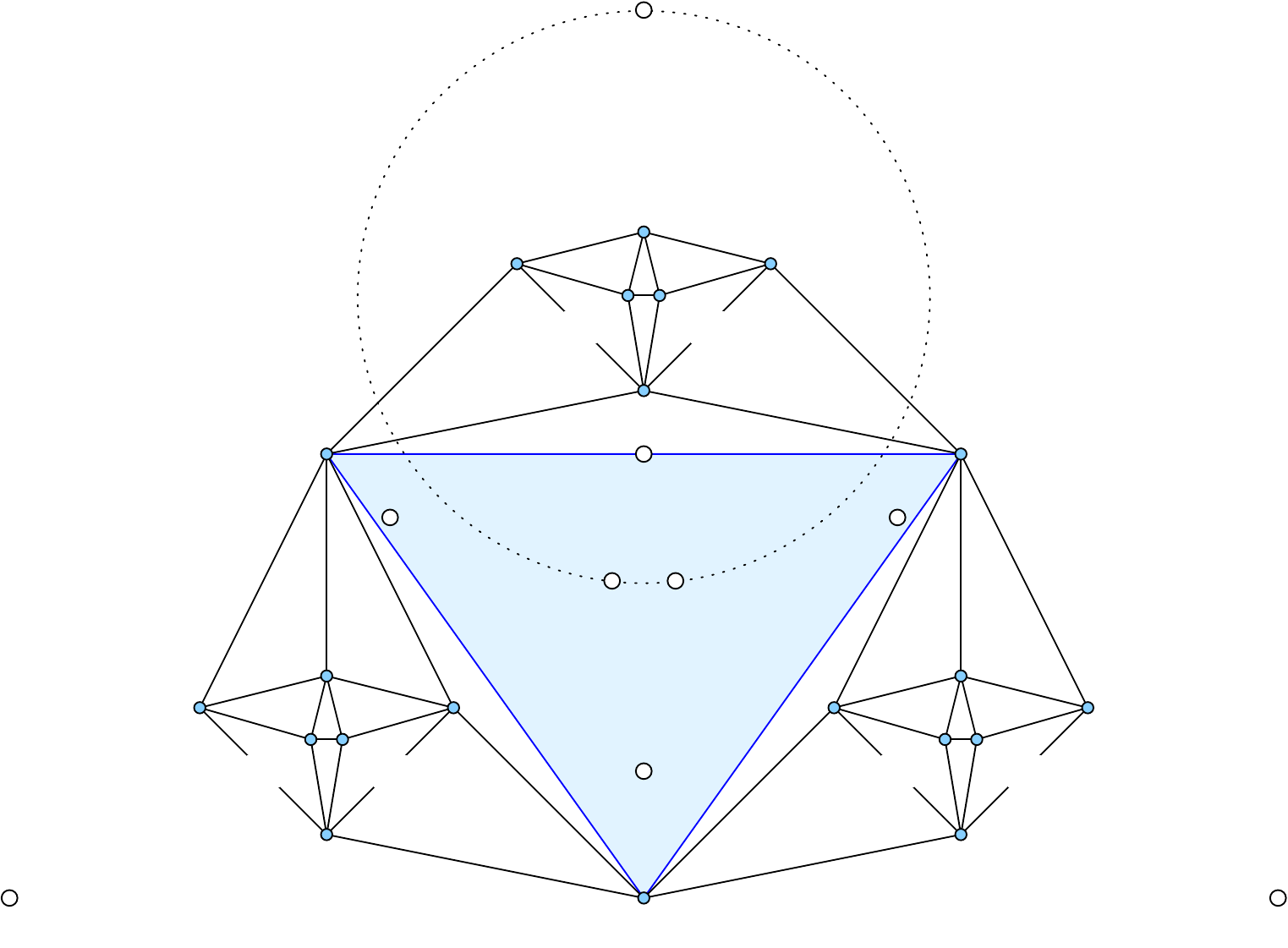}
  \caption{The minimum angle in the displayed level-$2$ hypertriangulation is larger than the minimum angle of the order-2 Delaunay triangulation of the same points.
  Indeed, the smallest angle in the hypertriangulation of about $9$ degrees is defined by the vertices $\Label{eh}, \Label{dh}, \Label{gh}$.
  For comparison, the circle in the picture proves that the angle of about $6.4$ degrees defined by the vertices $\Label{bc}, \Label{cd}, \Label{ac}$ belongs to the order-$2$ Delaunay triangulation (not shown).}
  \label{fig:min_counterexample}
\end{figure}

\subsection{Corollary for MaxMin Angle}
\label{sec:4.5}

Theorem~\ref{thm:angle_vector_optimality} implies that among all complete level-$2$ hypertriangulation, the order-$2$ Delaunay triangulation is distinguished by maximizing the minimum angle.
Using Sibson's result for level-$1$ hypertriangulations \cite{Sib78}, there is a short proof of this corollary.
No such similarly short proof is known for the angle vector optimality of order-$2$ Delaunay triangulations.
\begin{corollary}[MaxMin Angle Optimality]
  \label{cor:maxmin_angle_optimality}
  Let $A \subseteq \Rspace^2$ be finite and generic, and $P$ a complete triangulation of $A$.
  Then the minimum angle of the triangles in $\Phi_2 = f(P)$ is smaller than or equal to the minimum angle of the triangles in $\Delta_2 = f(\Delaunay{}{A})$.
\end{corollary}
\begin{proof}
  Write $D = \Delaunay{}{A}$, for each $x \in A$, write $D(x) = \Delaunay{}{A \setminus \{x\}}$, and let $P(x)$ be the triangulation of $A \setminus \{x\}$ obtained by removing the triangles that share $x$ from $P$ and adding the triangles in the constrained Delaunay triangulation of $\white{P,x}$.
  By Sibson's theorem, the smallest angle in $P$ is smaller than or equal to the smallest angle in $D$, and for each $x \in A$, the smallest angle in $P(x)$ is smaller than or equal to the smallest angle in $D(x)$.
  The smallest angle in $\Delta_2$ is the minimum angle in $D$ and all $D(x)$, and the smallest angle in $\Phi_2$ is the minimum angle in $P$ and all $P(x)$, for $x \in A$.
  Hence, the smallest angle in $\Phi_2$ is smaller than or equal to the smallest angle in $\Delta_2$.
\end{proof}

\section{Uniqueness of Local Angle Property}
\label{sec:5}

In this section, we prove the second main result of this paper, which supports the Local Angle Conjecture formulated at the end of Section~\ref{sec:3.3} by proving it for the case $k = 2$.
We begin with three basic lemmas on hypertriangulations that satisfy some or all of the conditions in Definition~\ref{def:local_angle_property}.

\subsection{Useful Lemmas}
\label{sec:5.1}

To streamline the discussion, we call a union of black triangles a \emph{black region} if its interior is connected and it is not contained in a larger black region of the same triangulation.
Similarly, we define \emph{white regions}.
Furthermore, we refer to \emph{black} or \emph{white angles} when we talk about the angles inside a black or white triangle.
\begin{lemma}[Black Regions are Convex]
  \label{lem:black_regions_are_convex}
  Let $A \subseteq \Rspace^2$ be finite and generic, and let $P_k$ be a level-$k$ hypertriangulation of $A$ that satisfies \textsc{(bb)}.
  Then every black region of $P_k$ is convex, and all vertices of the restriction of $P_k$ to the black region lie on the boundary of that region.
\end{lemma}
\begin{proof}
  Let $a$ be a boundary vertex of a black region, with edges $ab_0$, $ab_1, \ldots, ab_{p+1}$ bounding the $p+1$ incident black triangles in the region.
  \textsc{(bb)} implies $\measuredangle ab_{i-1}b_i + \measuredangle ab_{i+1}b_i > \pi$ for $1 \leq i \leq p$, so the sum of the $2(p+1)$ angles is larger than $p \pi$.
  Hence, the sum of the remaining $p+1$ angles at $a$ is less than $\pi$, as required for the black region to be convex at $a$.
  The same calculation shows that a ring of black triangles around a vertex in the interior of the black region is not possible.
  {\sloppy
  
  }
\end{proof}

\begin{lemma}[Total Black Angles]
\label{lem:total_black_angles}
  Let $A \subseteq \Rspace^2$ be finite and generic, and let $P_k$ be a level-$k$ hypertriangulation of $A$ that has the local angle property.
  Then the sum of black angles at any vertex of $P_k$ is less than $\pi$.
\end{lemma}
\begin{proof}
  Let $a$ be a vertex of $P_k$. 
  If $a$ is a boundary vertex, then the claim is trivial. 
  If $a$ is an interior vertex and incident to at most one black region, then the claim follows from Lemma~\ref{lem:black_regions_are_convex}.
  So assume that $a$ is interior and incident to $p \geq 2$ black and therefore the same number of white regions.
  Let $ab_1, ab_2, \ldots, ab_{2p}$ be the edges separating the black and white regions around $a$, with the region between $ab_1$ and $ab_2$ being black. 
  We also assume that the angle between any two consecutive edges is less than $\pi$, else the claim is obvious.

  \smallskip
  We look at the edge $ab_2$ and claim that $\measuredangle ab_1b_2 > \measuredangle ab_3b_2$.
  The black region between $ab_1$ and $ab_2$ satisfies \textsc{(bb)}, so its triangulation is the farthest-point Delaunay triangulation.
  In it, every triangle that shares an edge with the boundary of the region has the property that the angle opposite to the boundary edge is minimal over all choices of third vertex \cite{Epp92}.
  Therefore, $\measuredangle ab_1b_2$ is greater than or equal to the angle opposite to $ab_2$ inside the black triangle.

  Similarly, the triangulation of the white region between $ab_2$ and $ab_3$ satisfies \textsc{(ww)}, so its triangulation is the constrained Delaunay triangulation of the region. 
  Thus, $\measuredangle ab_3b_2$ is smaller than or equal to the angle opposite to $ab_2$ inside the white triangle.
  Applying \textsc{(bw)} to $ab_2$, we get the claimed inequality.

  \smallskip
  We repeat the same argument for all other edges separating black from white regions around $a$, and compare the sum of black and white angles opposite these edges:
  \begin{align}
    \sum\nolimits_{i=0}^p \left( \measuredangle ab_{2i+1}b_{2i+2} + \measuredangle ab_{2i+2}b_{2i+1} \right)
    &>  \sum\nolimits_{i=0}^p \left( \measuredangle ab_{2i}b_{2i+1} + \measuredangle ab_{2i+1}b_{2i} \right) ,
    \label{eqn:bwanglesums}
  \end{align}
  in which the indices are modulo $2p$.
  The sum of black angles at $a$ is $p\pi$ minus the first sum in \eqref{eqn:bwanglesums}, and the sum of white angles at $a$ is $p \pi$ minus the second sum in \eqref{eqn:bwanglesums}.
  Therefore the sum of black angles at $a$ is less then the sum of white angles at $a$.    
\end{proof}

\begin{lemma}[Local Angle Property and Aging Function]
\label{lem:local_and_aging}
  Let $A \subseteq \Rspace^2$ be finite and generic, $P_k$ a level-$k$ hypertriangulation of $A$, and $P_{k-1} = F^{-1} (\Black{P_k})$ a level-$(k-1)$ hypertriangulation of $A$.
  If $P_k$ has the local angle property, then $P_{k-1}$ satisfies \textsc{(ww)}.
\end{lemma}

\begin{proof}
  We consider two adjacent white triangles with vertices $\Label{X\!a}$, $\Label{X\!b}$, $\Label{X\!c}$ and $\Label{X\!b}$, $\Label{X\!c}$, $\Label{X\!d}$ in $P_{k-1}$.
  Applying the aging function, we get two black triangles of $P_k$ with vertices $\Label{X\!ab}$, $\Label{X\!ac}$, $\Label{X\!bc}$ and $\Label{X\!bc}$, $\Label{X\!bd}$, $\Label{X\!cd}$. 
  They share $\Label{X\!bc}$, which implies that the sum of their angles at this vertex is less than $\pi$ by Lemma~\ref{lem:total_black_angles}.
  The two black triangles are homothetic copies of $abc$ and $bcd$, and so are the corresponding two white triangles in $P_{k-1}$, so
  \textsc{(ww)} follows.
\end{proof}

\subsection{Level-2 Hypertriangulations}
\label{sec:5.2}

We are now ready to confirm the Local Angle Conjecture for level-$2$ hypertriangulations.
\begin{theorem}[Local Angle Conjecture for Level 2]
  \label{thm:level-2_hypertriangulations_have_lap}
  Let $A \subseteq \Rspace^2$ be finite and generic, and let $P_2$ be a maximal level-$2$ hypertriangulation of $A$. 
  Then $P_2$ has the local angle property iff it is the order-$2$ Delaunay triangulation of $A$.
\end{theorem}
\begin{proof}
  No two black triangles in $P_2$ share an edge, which implies that \textsc{(bb)} is void.
  On the other hand, there are pairs of adjacent white triangles that belong to the triangulation of white regions in $P_2$.
  In complete level-$2$ hypertriangulations, each such region is a polygon without points (vertices) inside, but in the more general case of maximal level-$2$ hypertriangulations considered here, there may be such points or vertices.
  In either case, \textsc{(ww)} implies that the restriction of $P_2$ to each white region is the constrained Delaunay triangulation of this region.

  \smallskip
  Let $P$ be the underlying (order-1) triangulation of $A$, which consists of the images of the black triangles in $P_2$ under the inverse aging function.
  We begin by establishing that $P$ is maximal and therefore $P_2$ is complete. 
  Suppose $x \in A$ is not a vertex of $P$, and let $abc$ be the triangle in $P$ that contains $x$ in its interior.
  Consider the triangle with vertices $c' = \Label{ab}$, $b' = \Label{ac}$, and $a' = \Label{bc}$ in $\Black{P_2}$. 
  The edge connecting $b'$ and $c'$ is shared with $\Label{\white{P_2,a}}$, and this white region contains $x' = \Label{ax}$.
  Since $P_2$ is maximal, by assumption, $x'$ is a vertex of the restriction of $P_2$ to this white region.
  Let $d'$ be the vertex of $P_2$ such that $b'd'c'$ is a triangle of $\White{P_2,a}$. 
  This triangle belongs to the constrained Delaunay triangulation of the white region, because this is the only triangulation that satisfies the \textsc{(ww)} property for all interior edges. 
  Recall that the triangle $b'yc'$ in the constrained Delaunay triangulation has the property that the angle at $y$ is maximal over all possible choices of $y\in A$ in the region visible from $b'$ and $c'$.
  Hence, $\measuredangle b'd'c' \geq \measuredangle b'x'c'$, as $x'$ is visible from $b'$ and $c'$, because the whole triangle $b'x'c'$ lies in the white region,
  but also $\measuredangle b'x'c' = \measuredangle bxc > \measuredangle bac = \measuredangle b'a'c'$, because $x$ is inside $abc$.
  This implies $\measuredangle b'd'c' > \measuredangle b'a'c'$, which contradicts \textsc{(bw)} for $P_2$, so $P$ is necessarily maximal.
  
  \Skip{Recall that the triangle $b'd'c'$ in the constrained Delaunay triangulation of the white region has the property that the angle at $d'$ is maximal over all possible choices of $d'$ visible from $b'$ and $c'$.
  Hence, $\measuredangle b'd'c' \geq \measuredangle b'x'c'$,
  but also $\measuredangle b'x'c' = \measuredangle bxc > \measuredangle bac = \measuredangle b'a'c'$ because $x$ is inside $abc$.
  This implies $\measuredangle b'd'c' > \measuredangle b'a'c'$, which contradicts \textsc{(bw)} for $P_2$, so $P$ is necessarily maximal.}

  \smallskip
  Applying Lemma \ref{lem:local_and_aging} to $P_2$, we conclude that $P$ satisfies \textsc{(ww)}. 
  Since $P$ is maximal, the only choice left is that $P$ is the Delaunay triangulation of $A$.
  The black triangles in $P_2$ thus coincide with the black triangles in the order-$2$ Delaunay triangulation of $A$, and $P_2$ restricted to each of its white regions is the constrained Delaunay triangulation of this region.
  Hence, $P_2$ is the order-$2$ Delaunay triangulation of $A$. 
\end{proof}

\subsection{Level-3 Hypertriangulations}
\label{sec:5.3}

We say $A \subseteq \Rspace^2$ is in \emph{convex position} if all its points are vertices of $\conv{A}$.
For such sets, we can extend Theorem~\ref{thm:level-2_hypertriangulations_have_lap} to level-$3$ hypertriangulations.
The main differences to general finite sets are that all triangulations have the same number of triangles, and the aging function exists, as established by Galashin in \cite{Gal18} but see also \cite{EGGHS23}.
We use this function together with the characterization of the order-$2$ Delaunay triangulation as the only level-$2$ hypertriangulation
that has the local angle property.
\begin{theorem}[Local Angle Conjecture for Level 3]
  \label{thm:level-3_hypertriangulations_and_convex_position}
  Let $A \subseteq \Rspace^2$ be finite, generic, and in convex position, and let $P_3$ be a hypertriangulation of $A$.
  Then $P_3$ has the local angle property iff it is the order-$3$ Delaunay triangulation of $A$.
\end{theorem}
\begin{proof}
  By Theorem~\ref{thm:order-k_Delaunay_triangulations_have_lap}, the order-$3$ Delaunay triangulation has the local angle property.
  Let $P_3$ be a possibly different level-$3$ hypertriangulation that also has the local angle property, and let $P_2 = F^{-1} (\Black{P_3})$, which exists because $A$ is in convex position \cite{Gal18}.
  By Lemma~\ref{lem:local_and_aging}, $P_2$ satisfies \textsc{(ww)}.
  Recall that \textsc{(bb)} is void for level-$2$ hypertriangulations, so if in addition to \textsc{(ww)}, $P_2$ also satisfies \textsc{(bw)}, then it has the local angle property.
  By Theorem~\ref{thm:level-2_hypertriangulations_have_lap}, this implies that $P_2$ is the order-$2$ Delaunay triangulation of $A$.
  Its white triangles are in bijection with the triplets of points whose circumcircles enclose exactly one point of $A$, and since $\Black{P_3} = F(\White{P_2})$, so are the black triangles of $P_3$.
  Thus, $P_3$ has the same black triangles as the order-$3$ Delaunay triangulation of $A$.
  Furthermore, the white regions of $P_3$ coincide with the white regions of the order-$3$ Delaunay triangulation, and because the restriction of either triangulation to a white region is the constrained Delaunay triangulation of that region, we conclude that $P_3$ \emph{is} the order-$3$ Delaunay triangulation of $A$.
  \smallskip

  \begin{figure}[hbt]
    \centering
    \vspace{0.0in}
    \resizebox{!}{2.6in}{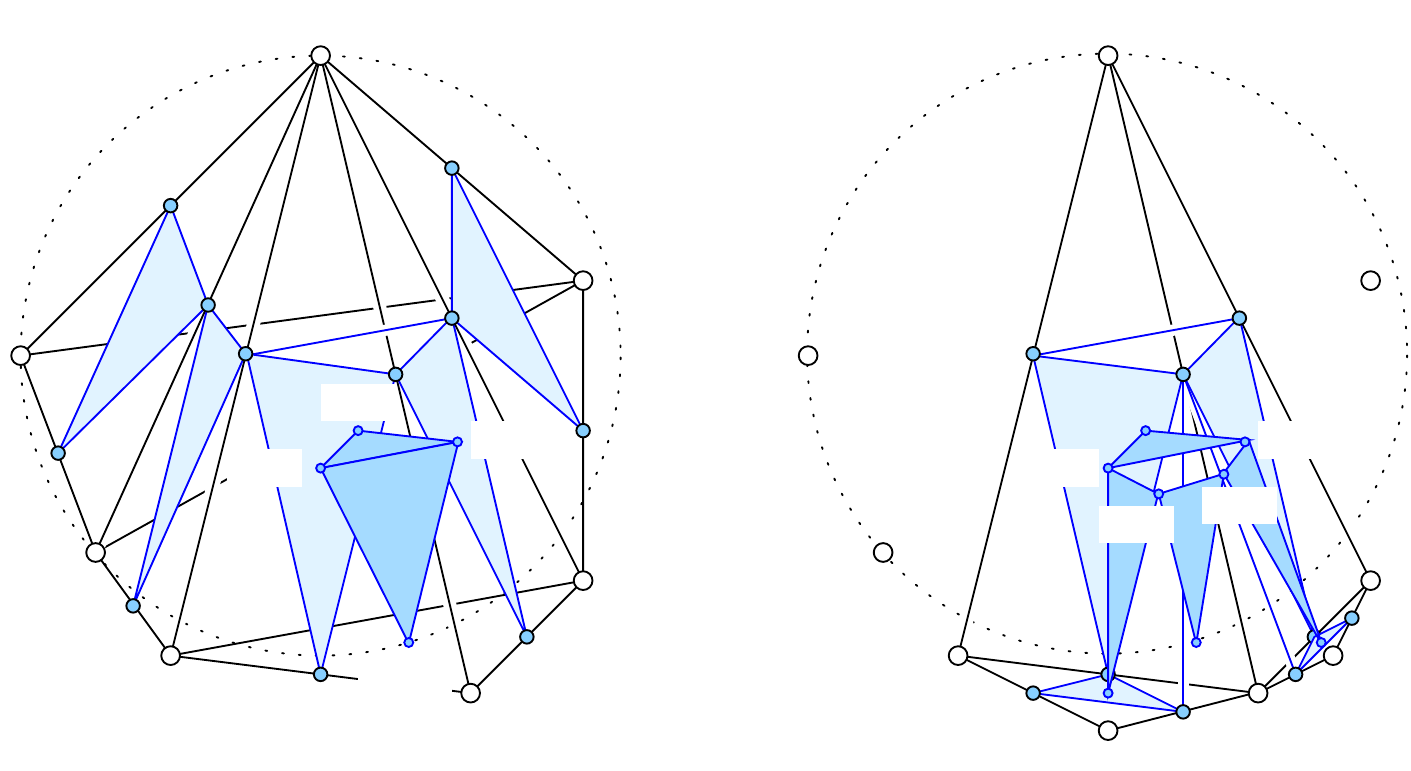}
    \vspace{-0.0in}
    \caption{
    The superposition of three levels.
    \emph{Left:} part of the star of $a$ in $P$ on level $1$, the (white) triangles in this star aging to black triangles in $P_2$ on level $2$, and the only two white triangles in the star of $\Label{av}$ aging to two black triangles in $P_3$ on level $3$.
    One is similar to $uvw$ and the other to $auw$, which is assumed to be unique.
    \emph{Right:} compared to the configuration on the \emph{left}, there are two extra white triangles, which increase the star of $\Label{av}$ in $P_2$ from two to four triangles.
    Accordingly, we see a white quadrangle on level $3$.}
    \label{fig:convexposition}
  \end{figure}
  It remains to show that $P_2$ indeed satisfies \textsc{(bw)}.
  To derive a contradiction, we assume it does not.
  Let $\Label{ab}$, $\Label{ac}$, $\Label{bc}$ and $\Label{ab}$, $\Label{ac}$, $\Label{ad}$ be the vertices of a black triangle and an adjacent white triangle that violate \textsc{(bw)}, so $\measuredangle{bac} < \measuredangle bdc$.
  Let $P = F^{-1} (\Black{P_2})$, and consider the star of $a$ in $P$.
  All vertices are in convex position, including $a, b, c, d$, so we may assume that $ac$ crosses $bd$, as in Figure~\ref{fig:convexposition} on the left.
  Let $ax_1 = ab, ax_2 = ac, \ldots, ax_p = ad$ be the sequence of edges in the star of $a$ that intersect $bd$.
  We consider the polygon with vertices $a, x_1, x_2, \ldots, x_p$.
  Since $A$ is in convex position, the polygon is convex, which implies that its constrained Delaunay triangulation is also the Delaunay triangulation of the $p+1$ points.
  Denote this Delaunay triangulation by $\Delta$, and note that it includes $bcd = x_1x_2x_p$:  $a$ is outside the circumcircle of $bcd$, because $abc$ and $bcd$ violate \textsc{(bw)}, and so is every $x_i$ with $3 \leq i \leq p-1$, because $bcd$ is a triangle in $\White{P_2, a}$.
  The rest of $\Delta$ consists of $abd = ax_1x_p$ and the triangles of $\White{P_2, a}$ on the other side of $x_2 x_p$.
  An \emph{ear} of $\Delta$ is a triangle that has two of its edges in the boundary of the polygon.
  For example, $ax_1x_p$ is an ear, but since every triangulation of a polygon with at least four vertices has at least two ears, there is another one, and we write $uvw = x_{i-1}x_ix_{i+1}$ for a second ear of $\Delta$.
  The corresponding triangle in $P_2$ has vertices $\Label{au}$, $\Label{av}$, $\Label{aw}$ and is adjacent to black triangles with vertices $\Label{au}$, $\Label{av}$, $\Label{uv}$ and $\Label{av}$, $\Label{aw}$, $\Label{vw}$.
  Both pairs violate \textsc{(bw)} because $a$ lies outside the circumcircle of $uvw$.
  Looking closely at this configuration, we note that $\Label{av}$ is shared by the two black triangles and also belongs to $\Label{\white{P_2,a}}$ and $\Label{\white{P_2,v}}$; see again Figure~\ref{fig:convexposition} on the left.
  We distinguish between two cases: when $\Label{av}$ belongs to only one triangle in the triangulation of the latter white region, and when it belongs to two or more such triangles.

  Assuming the first case, we apply the aging function to the two white triangles sharing $[av]$, which gives two black triangles with vertices $\Label{auv}$, $\Label{auw}$, $\Label{awv}$ and $\Label{auv}$, $\Label{awv}$, $\Label{uwv}$ in $P_3$.
  They share an edge, and since $a$ lies outside the circumcircle of $uvw$, they violate \textsc{(bb)}, which is the desired contradiction.

  \smallskip
  There is still the second case, when $\Label{av}$ belongs to two or more triangles in the triangulation of $\Label{\white{P_2,v}}$.
  Let $\Label{uv} = \Label{y_1v}, \Label{y_2v}, \ldots, \Label{y_qv} = \Label{wv}$ be the vertices of $\Label{\white{P_2,v}}$ connected to $\Label{av}$; see Figure~\ref{fig:convexposition} on the right.
  These $q$ edges bound $q-1$ white triangles in $P_2$.
  Consider their images under the aging function, which are $q-1$ black triangles in $P_3$.
  Together with the black triangle with vertices $\Label{auv}$, $\Label{auw}$, $\Label{awv}$, these black triangles surround a convex $q$-gon with vertices $\Label{auv} = \Label{ay_1v}, \Label{ay_2v}, \ldots, \Label{ay_qv} = \Label{awv}$; see again Figure~\ref{fig:convexposition} on the right.
  The $q$-gon is convex because $A$ is in convex position, and we claim it is a white region in $P_3$.
  If there is any black triangle, $T$, inside this $q$-gon, then we consider any generic segment connecting $T$ to the boundary of the $q$-gon, and the closest part of that segment to the boundary colored black in $P_3$.
  By construction, the triangle $T'$ containing this part has two vertices labeled $\Label{avz_1}$ and $\Label{avz_2}$, for some $z_1$ and $z_2$. 
  Hence, $F^{-1}(T')$ is a white triangle of $P_2$ incident to $\Label{av}$, which is impossible, as all white triangles in $P_2$ incident to $\Label{av}$ age to black triangles surrounding the $q$-gon.
  Recall that $P_3$ satisfies \textsc{(ww)}, so the restriction of $P_3$ to the $q$-gon is the (constrained) Delaunay triangulation of the $q$-gon.

  Consider the edge connecting $\Label{auv} = \Label{ay_1v}$ and $\Label{awv} = \Label{ay_qv}$ of the $q$-gon, and let $\Label{ay_iv}$ be the third vertex of the incident white triangle.
  Because this triangle is part of the (constrained) Delaunay triangulation, we have $\measuredangle uy_jw < \measuredangle uy_iw$ for all $j \neq i$, and because $P_3$ satisfies \textsc{(bw)}, we have $\measuredangle uy_iw < \measuredangle uvw$.
  Recall that $a$ lies outside the circumcircle of $uvw$, so $\measuredangle uvw + \measuredangle uaw < \pi$.
  This implies $\measuredangle uy_iw + \measuredangle uaw < \pi$.
  Hence, the circumcircle of the triangle with vertices $\Label{uv}, \Label{y_iv}, \Label{wv}$ does not enclose any of the other vertices.
  It follows that the triangle belongs to the constrained Delaunay triangulation of the polygon with vertices $\Label{uv} = \Label{y_1v}, \Label{y_2v}, \ldots, \Label{y_qv} = \Label{wv}$, but it does not because this polygon is triangulated with edges that all share $\Label{av}$.
  This gives the final contradiction.
\end{proof}

\section{Concluding Remarks}
\label{sec:6}

In this last section, we discuss open questions about hypertriangulations.
The obvious one is whether optimality properties other than angles can be generalized from level $1$ to higher levels: for example the smallest circumcircle \cite{DASi89}, the smallest enclosing circle \cite{Raj94}, roughness \cite{Rip90}, and other functionals \cite[Chapter 3]{DRS10} and \cite{Mus03}, which are all optimized by the order-$1$ Delaunay triangulation.
In addition, we list a small number of more specific questions and conjectures directly related to the discussions in the technical sections of this paper.

\medskip \noindent \textbf{Flipping as a proof technique.}
Sibson's original proof for the angle vector optimality of the Delaunay triangulation \cite{Sib78} uses the sequence of edge-flips provided by Lawson's algorithm~\cite{Law77}.
There is such a sequence for every complete triangulation, and each flip lexicographically increases the vector.
The authors of this paper pursued a similar approach to prove Theorem~\ref{thm:angle_vector_optimality} using the flips of Types~I to IV developed in \cite{EGGHS23}; see Figure~\ref{fig:no-compound} on the right.
While these flips connect all level-$2$ hypertriangulations of a finite generic set (Theorem~4.4 in \cite{EGGHS23}), they do not necessarily lexicographically increase the angle vector.

Indeed, there is a level-$2$ hypertriangulation of six points, $Q_2$, different from the order-$2$ Delaunay triangulation, such that every applicable flip lexicographically decreases the sorted angle vector.
The six points in this example are $a, b, c, g, h, i$ in Figure~\ref{fig:no-compound}, and we obtain $Q_2$ from the shown hypertriangulation by removing the vertices $\Label{ad}, \Label{dg}, \Label{be}, \Label{eh}, \Label{cf}, \Label{fi}$.
In $Q_2$, there are only three possible flips, all of Type~I, and all three lexicographically decrease the sorted angle vector.
Incidentally, six is the smallest number of points for which such a counterexample to using flips as a proof technique for level-$2$ hypertriangulations exists.
\begin{figure}[hbt]
  \centering
  \resizebox{!}{2.5in}{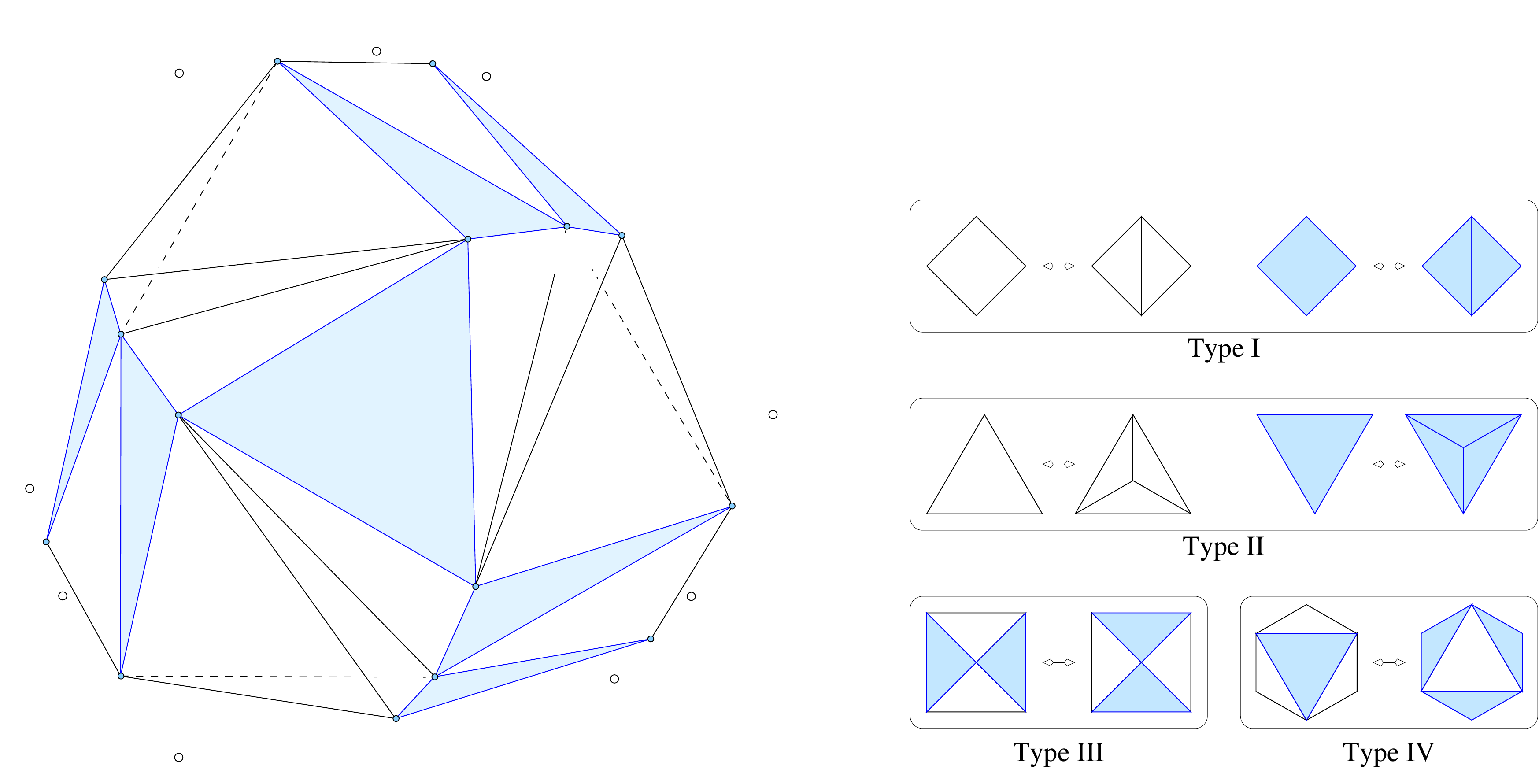}
  \caption{\emph{Right:} the four types of flips that connect the level-$2$ hypertriangulations of a given set. 
  \emph{Left:} a complete level-$2$ hypertriangulation such that every applicable compound flip decreases the sorted angle vector.
  The \emph{dashed} edges appear after removing vertices $\Label{ad}, \Label{dg}, \Label{be}, \Label{eh}, \Label{cf}, \Label{fi}$.}
  \label{fig:no-compound}
\end{figure}

Let $P_2$ be the level-$2$ hypertriangulation in Figure~\ref{fig:no-compound} (without removing points $d, e, f$).
It provides a counterexample to using a local retriangulation operation more powerful than a flip as a proof technique.
To explain, let $P$ and $P'$ be two complete level-$1$ hypertriangulations of the same set.
Let $P_2 = F(P)$ and $P_2' = F(P')$ be the aged level-$2$ hypertriangulations such that the restriction to any white region is the constrained Delaunay triangulation of that region.
Equivalently, $P_2$ and $P_2'$ satisfy \textsc{(ww)}.
If $P$ and $P'$ are connected by a single flip of Type~I, we say that $P_2$ and $P_2'$ are connected by a \emph{compound flip}.
It consists of a sequence of Type~I flips affecting white regions in $P_2$, followed by a Type~III flip, followed by a sequence of Type~I flips affecting white regions in $P_2'$.
Such a compound flip may increase the sorted angle vector even if some of its elementary flips do not.
Nevertheless, all compound flips applicable to $P_2$ in Figure~\ref{fig:no-compound} decrease the sorted angle vector, thus spoiling the hope for an elegant proof of Theorem~\ref{thm:angle_vector_optimality} using compound flips.
This motivates the following question.
\begin{question}
  \label{qu:A}
  Does there exist a flip-like approach to proving Theorem~\ref{thm:angle_vector_optimality} on the angle vector optimality for complete level-$2$ hypertriangulations?
\end{question}

\medskip \noindent \textbf{Angle vector optimality and local angle property.}
Recall that Theorem~\ref{thm:angle_vector_optimality} proves the optimality of the Delaunay triangulation only for order-$2$ and among all complete level-$2$ hypertriangulations.
Indeed, Section~\ref{sec:4.4} shows counterexamples for order-$3$ and for relaxing to maximal level-$2$ hypertriangulations.
This motivates the following two questions:
\begin{itemize}
  \item Is there a sense in which the order-$k$ Delaunay triangulations optimize angles for all $k$?
  \item Among all maximal level-$2$ hypertriangulations, which one lexicographically maximizes the sorted angle vector?
\end{itemize}
Recall also that Theorem~\ref{thm:level-2_hypertriangulations_have_lap} proves that the local angle property characterizes the order-$2$ Delaunay triangulation among all maximal level-$2$ hypertriangulations, leaving the case of higher orders open.
We venture the following conjecture, while keeping in mind that some condition on the family of competing hypertriangulations is needed to avoid Delaunay triangulations of proper subsets of the given points.
\begin{conjecture_new}[Local Angle Conjecture]
  \label{conj:B}
  Let $A \subseteq \Rspace^2$ be finite and generic, and for every $1 \leq k \leq \card{A}-1$ let $\Fcal_k$ be the family of level-$k$ hypertriangulations that have the local angle property.
  Then $P_k \in \Fcal_k$ has the maximum number of triangles iff $P_k$ is the order-$k$ Delaunay triangulation of $A$.
\end{conjecture_new}
In the formulation of this conjecture, we maximize the number of triangles over all members of $\Fcal_k$, and not over all level-$k$ hypertriangulations of $A$, because the latter may not contain any that have the local angle property.
To see this, let $A$ be any finite set that is not in convex position.
For $k = \card{A}-1$, all triangles are black, and by Lemma~\ref{lem:black_regions_are_convex}, condition \textsc{(bb)} of the local angle property implies that no point in the interior of $\conv{A}$ is a vertex of the triangulation.
Thus every hypertriangulation on this level that has the local angle property does not have the maximum number of triangles.
Also note that Theorem~\ref{thm:level-3_hypertriangulations_and_convex_position} shows that the conjecture holds for the case $k=3$ and points in convex position.
More generally, for such points all level-$k$ hypertriangulations have the same number of triangles; see \cite{EGGHS23} for interpretation of results from \cite{Gal18,Pos06}.

\medskip \noindent \textbf{Maximal and maximum hypertriangulations.}
Recall that a hypertriangulation is \emph{maximal} if no other hypertriangulation of the same level subdivides it.
We say a hypertriangulation is \emph{maximum} if no other hypertriangulation of the same level has more triangles.
In an attempt to generalize Lemma~\ref{lem:complete_implies_maximal} to levels beyond $2$, we conjecture that the number of triangles in a maximal hypertriangulation depends on the given points but not on how these points are triangulated.
\begin{conjecture_new}
  \label{conj:C}
  Let $A \subseteq \Rspace^2$ be finite and generic.
  Then any two maximal level-$k$ hypertriangulations of $A$ have the same and therefore maximum number of triangles.
  In other words, every maximal level-$k$ hypertriangulation is maximum.
\end{conjecture_new}
The conjecture holds for points in convex position \cite{Gal18,Pos06}, and we have verified it for a few small configurations in non-convex position.
If true, this might have combinatorial meaning as the vertices of maximal hypertriangulations would then encode data from the matroid defined by the point set. 
We refer to \cite{GP17} for an extensive discussion of this topic in connection to zonotopal tilings and collections of separated subsets, in particular for points in convex position.


\end{document}